\documentclass[12pt]{article}
\usepackage{amsmath,amsthm}
\usepackage{url}
\usepackage{times}
\usepackage{graphicx}
\usepackage{color}
\usepackage{amssymb}
\usepackage{amsmath,empheq}
\usepackage{hyperref}

\newtheorem{theorem}{Theorem}
\newtheorem{remark}[theorem]{Remark}
\newtheorem{corollary}[theorem]{Corollary}
\newtheorem{lemma}[theorem]{Lemma}
\newtheorem{definition}[theorem]{Definition}
\newtheorem{proposition}[theorem]{Proposition}

\newcommand{\comment}[1]{}

\usepackage{amssymb}
\usepackage{amsfonts}
\usepackage{times}

\makeatletter
\renewcommand*\env@matrix[1][*\c@MaxMatrixCols c]{%
	\hskip -\arraycolsep
	\let\@ifnextchar\new@ifnextchar
	\array{#1}}
\makeatother

\definecolor{gold}{rgb}{0.83, 0.69, 0.22}

\definecolor{greenncs}{rgb}{0.0,0.62,0.42}

\usepackage{makeidx}
\makeindex

\begin{document}

	\title{Characterization of classical orthogonal polynomials in two continuous variables}
	\author{M. KENFACK NANGHO$^a$, K. JORDAAN$^b$ and B. JIEJIP NKWAMOUO$^c$\\
		\\
		$^{a,c}$Department of Mathematics and Computer Science,\\ University of Dschang, PO Box 67, Dschang, Cameroon\\
		$^b$ Department of Decision Sciences,\\ University of South Africa, PO Box 392, Pretoria, 0003, South Africa}
	\maketitle
	
	\begin{abstract}
		\noindent For a family of  polynomials in two {continuous} variables, orthogonal with respect to a weight function, we prove, under {suitable} conditions, {the equivalence of the following properties}: the matrix Pearson equation of the weight, the second order linear partial differential equation, the orthogonality of the gradients, the matrix Rodrigues formula involving tensor product{s} of matrices, and the so-called first structure relation. We then introduce a notion of classical orthogonal polynomials in two variables and relate the corresponding theory for weight functions and moment functionals. Finally, we present a nontrivial example that illustrates and delineates our contribution to the field.	
	\end{abstract}
	
	\begin{itemize}
		\item[]\textbf{Keywords}: Orthogonal polynomials in two {continuous} variables, Characterization, Partial Differential Equations.
		\item[]\textbf{Mathematics Subject Classification (2020)}: Primary 42C05; Secondary 33C50
	\end{itemize}

	\section{Introduction and Notations}
	\subsection{Introduction}
	Let $\left\{p_n\right\}_{n}$ be a family  of polynomials, orthogonal with respect to a weight function, $\rho$, on an interval {$(a,b)$} of $\mathbb{R}$ (set of real numbers). $\left\{p_n\right\}_n$ is classical if and only if its derivatives are also orthogonal. This property leads to  the existence of a sequence of real numbers $\mu_{n,m-1}$, $n\in \mathbb{N}$, $0\leq m\leq n$ and a polynomial $\phi$ of degree at most 2 such that \cite[Eq. (1.3.5)]{NIKI}
	\begin{equation*}\label{e1}
		\left[\rho(x) \Phi^{m}p_n^{(m)}\right]^{(1)}(x)=-\mu_{n,m-1}\rho(x)\Phi^{m-1}(x)p_n^{(m-1)}(x),
	\end{equation*}
	
	where $p_n^{(m)}=\frac{d^m}{dx^m}p_n$. Iterating and taking $m=n$, one has 
	
	\begin{equation}\label{e2}
		p_n(x)=\frac{(-1)^np_n^{(n)}(x)}{\rho(x)\prod_{j=0}^{n-1}\mu_{n,j}}\left[\rho\,\Phi^{n}\right]^{(n)}(x)
	\end{equation}
	which is the Rodrigues formula. Characterization of classical orthogonal polynomials in one variable is obtained by proving equivalence between (\ref{e2}), the Pearson equation of the weight, orthogonality of derivatives,  a Sturm-Liouville type equation, the so-called first structure relation and a non-linear relation (see {\cite[§5]{Al-Salam 1990}}). These properties are very important when using classical orthogonal polynomials in probability, partial differential equations and mathematical physics (see \cite{Williams, Fornberg} and references therein). Since the work of Krall and Sheffer \cite{Krall} introducing classical polynomials in two variables as polynomial solutions to the linear partial differential equation
	\begin{eqnarray}\label{e3}
		&&(\alpha x^2+d_1x+e_1y+f_1)\partial_x^2Y+(2\alpha xy+d_2x+e_2y+f_2)\partial_x\partial_yY\nonumber \\
		&&+(\alpha y^2+d_3x+e_3y+f_3)\partial_y^2Y+(\delta\,x+h_1)\partial_xY+(\delta\, y+h_2)\partial_yY+\lambda_nY=0,
	\end{eqnarray}
	$\lambda_n\neq 0$ for $n\neq 0$, where $n$ is the degree of the polynomial solution, the
	theory and application of orthogonal polynomials in two variables has {attracted attention} from various domains of mathematics 
	\cite{Koornwinder, Yuan, Barkry2013}. In 1975, Koornwinder \cite{Koornwinder} studied examples of two-variable analogues of Jacobi polynomials, and he introduced seven classes of orthogonal polynomials which he considered to be bivariate analogues of Jacobi polynomials. Some of these polynomials are classical according to the Krall and Sheffer definition and others are not. Therefore the Krall and Sheffer classification seems to be incomplete. Lyskova \cite{Lyskova} studied conditions on the polynomial coefficients in (\ref{e3}) in such a way that the partial derivatives of orthogonal polynomial solutions satisfy a partial differential equation of the same type. He obtained that $e_1=d_3=0$, i.e. the polynomial coefficients of the partial derivatives in $x$ (respectively in $y$) depend only on $x$ (respectively on $y$). An analogue of Rodrigues' formula for Krall and Sheffer classical orthogonal polynomials in two variables has been obtained by Suetin \cite{Suetin}. In fact, for $n$ a positive integer, he defines
	
	\begin{equation*}
		p_{n-i,i}=\frac{1}{\rho}\partial_x^{n-i}\partial_y^i\left(p^{n-i}q^i\rho\right),
	\end{equation*}
	where $\rho(x,\,y)$ is a weight function over a simply connected domain and $p(x,\,y)$
	,$q(x,\,y)$ are polynomials related with the polynomial  coefficients of (\ref{e3}). However, in several cases, Rodrigues' formula (\ref{e2})  provides polynomials of total degree greater than $n$ (see \cite{Littlejohn}) and therefore it cannot be used to construct a basis of orthogonal polynomials in two variables. {\cite{Littlejohn} also connected (\ref{e3}) to Pearson's equation and derived a moment functional for families of orthogonal polynomials solutions of (\ref{e3}). In \cite{Lee2000}, the author characterized orthogonal polynomial solutions of equation \eqref{e3} whose partial derivatives with respect to $x$ are orthogonal. \cite{Kim1997, Kim1998} investigated  polynomial solutions of (\ref{e3}) which are orthogonal with respect to a moment functional.}  
	In \cite{Fern}, by using matrix notation of polynomials in several variables introduced in \cite{Kowalski2,Kowalski} and modified in \cite{Xu}, the authors extended the definition by Krall and Sheffer of classical orthogonal polynomials in two variables as matrix polynomial families satisfying a second order linear partial differential equation analogue to (\ref{e3}), whose coefficients are some polynomials of degree at most two without any restriction on their shape.  Using this notation, the {authors} of \cite{AMFPP, Miguel2009,Fern,FPP} defined classical functional as a regular moment functional $u$ satisfying the  Pearson-type equation 
	\begin{equation}
		div(\Phi u)=\Psi^tu, 	\varPhi=\left(\begin{matrix}
			\phi_{1,1} & \phi_{1,2}\\
			\phi_{2,1} & \phi_{2,2}
		\end{matrix}\right), \varPsi=\left(\begin{matrix}
			\psi_{1}\\
			\psi_{2}
		\end{matrix}\right),
	\end{equation}
	
	where $\phi_{i,j}, i,j=1,2$ and $\psi_l,\, l=1,2$ are polynomials in two variables of degree at most 2 and 1 respectively, and the determinant of $\langle u,\Phi\rangle$ is different from zero.  Considering classical orthogonal polynomials in two variables as family of vector polynomials $\{\mathbb{P}_n\}$  orthogonal with respect to a classical functional $u$ they proved that for $m\geq 0$ \cite{AMFPP}
	\begin{equation}\label{inieq}
		div\left[(\Phi\otimes I_{2^m})\otimes I_{n+m+1}\nabla^{(m+1)}\mathbb{P}_{n+m}\right]+\left[\tilde{\Psi}^t\otimes I_{n+m+1}\right]\nabla^{(m+1)}\mathbb{P}_{n+m}=\Lambda_{n+m}^m\nabla^{(m)}\mathbb{P}_{n+m},
	\end{equation}
	
	where $\nabla^{(m+1)}$ is the successive gradient $\nabla^{(m+1)}=\nabla(\nabla^{(m)})$.
	Moreover, considering the higher order differential operators $\nabla^{\{n\}}$ and $div^{\{n\}}$ acting over matrices by means of (cf. also \cite{Miguel2009})
	\begin{align}\label{mdiv}
		\nabla^{\{n\}}A&=\left(\mathcal{D}_0^nA^t,\mathcal{D}_1^nA^t,\dots \mathcal{D}_n^nA^t\right)^t\in \mathcal{M}_{(n+1)h,k}\left(\mathcal{P}\right),\\
		div^{\{n\}}(B_0^t,B_1^t,\dots,B_n^t)^t&=\sum_{i=0}^{n}\mathcal{D}^n_iB_i\in\mathcal{M}_{h,k}(\mathcal{P}),
	\end{align}
	where $\mathcal{D}^n_i={n\choose i}\partial_x^{n-i}\partial_y^i,\,A,\,B_i\in\mathcal{M}_{h,k}(\mathcal{P}),\,i=0,1,\dots n$, for $n=1$, $div^{\{1\}}$ is the usual divergence operator,
	they established that, if there exist $P_0,P_1$, 2 lines and 2 columns matrices which entries are polynomials of degree one, such that
	\begin{subequations}
		\begin{empheq}[left=\empheqlbrace]{align}
			\partial_x(\phi_{1,1}\varPhi)+\partial_y(\phi_{2,1}\partial_y\varPhi)=\varPhi P_0,\nonumber\\
			\partial_x(\phi_{1,2}\partial_x\varPhi)+\partial_y(\phi_{2,2}\varPhi)=\varPhi P_1\nonumber,
		\end{empheq}
	\end{subequations} then for $n\geq m\geq 0$:\\ There exist matrices $F_i^{n,m}\in\mathcal{M}_{(m+1)(i+1)(n+1)}(\mathbb{R})$ such that \cite{FPP}
	\begin{equation}\label{FSR1}
		\Phi^{\{m\}}\nabla^{\{m\}}\mathbb{P}_n^t=\sum_{i=n-m}^{n+m}(I_{m+1}\otimes\mathbb{P}_{i}^t)F_i^{n,m};
	\end{equation}
	$\{\nabla^{\{m\}}\mathbb{P}_n^t\}_{n\geq m}$ satisfy the orthogonality relation \cite{FPP} 
	\begin{equation}\label{or1}
		\langle	u,\,\left(\nabla^{\{m\}}\mathbb{P}_n^t\right)^t\Phi^{\{m\}}\nabla^{\{m\}}\mathbb{P}_j^t\rangle =0,\,n\neq j
	\end{equation}
	The family of polynomials \cite{Miguel2009}
	\begin{equation}\label{e4b}
		\mathbb{Q}^t_n=\frac{1}{\rho}div^{\{n\}}\left(\mathbb{\varPhi}^{\{n\}}\rho\right),
	\end{equation} 
	is orthogonal with respect to the moment functional $u$, that is $\mathbb{Q}_n$ is up to a matrix multiplicative factor equal to $\mathbb{P}_n$, 	where $\rho$ is the nontrivial function, $C^2$ in some open set and satisfying the matrix Pearson-type equation  
	\begin{equation}\label{e5a}
		div(\varPhi\,\rho)=\varPsi^t\,\rho, \varPhi=\left(\begin{matrix}
			\phi_{1,1} & \phi_{1,2}\\
			\phi_{2,1} & \phi_{2,2}
		\end{matrix}\right), \varPsi=\left(\begin{matrix}
			\psi_{1}\\
			\psi_{2}
		\end{matrix}\right)
	\end{equation}   
	In 2018, Marcell\'{a}n et al. \cite{Marcellan} proved that the moment functional is classical if and only if its moments satisfy two three-term relations.
	
	Although the study of orthogonal polynomials based on moment functional is more general and leads to several algebraic properties, it avoids functional analysis property highlighted by the weighted approach. This approach gives more analytical  information to the polynomial sequence, see \cite{DEIFT1999,Ismail2005,Al-Salam 1990,Van2006} for orthogonal polynomial in one variable. Analytical properties of weight functions, as the boundary condition, play key role in these works. For instance, classical orthogonal polynomials in one variable are known to be orthogonal with respect to  weight function $\rho$, supported on an interval $(a,\,b)$, solution to the Pearson-equation $(\phi\rho)'=\psi\rho$ under the boundary condition \begin{equation}\label{ee5}
		p\rho\phi|_a^b=0\;\text{for all $p$, polynomial},
	\end{equation} where $\phi$ and $\psi$ are polynomials of degree at most 2 and 1 respectively. In 1999 using this definition, Al-Salam  \cite[§5]{Al-Salam 1990} characterized classical orthogonal polynomials in one variable by means of six equivalent properties. The objective of this work is to state and prove  a similar theorem for orthogonal polynomials in two variables.    
	For a weight function $\rho$ on a simply connected open subset of $\mathbb{R}^2$,  we use the vector notation of polynomials in two variables, the  following extension of the boundary condition (\ref{ee5})
	\begin{equation}\label{bc0a}
		\lim\limits_{j\rightarrow\infty}1_{\partial\Omega_j}\left(\rho \Phi\,\nabla\,u\right)\cdot\overrightarrow{n_j}=0,\; \Omega_j=\Omega\cap B(O,j),\; {\rm for~all}\; u\in\mathcal{M}_{1,n}\left(\mathcal{P}\right)
	\end{equation}
	as well as the differential condition \begin{subequations}\label{dq0}
		\begin{empheq}[left=\empheqlbrace]{align}
			\phi_{1,1}\partial_x\varPhi+\phi_{2,1}\partial_y\varPhi=\varPhi\nabla\left(\phi_{1,1},\phi_{2,1}\right),\\
			\phi_{1,2}\partial_x\varPhi+\phi_{2,2}\partial_y\varPhi=\varPhi\nabla\left(\phi_{1,2},\phi_{2,2}\right),
		\end{empheq}
	\end{subequations}
	to state and prove that for a family $\{\mathbb{P}_n\}_n$ of monic polynomials, orthogonal with respect to $\rho$, satisfying the Pearson equation
	\begin{equation}\label{peaeq}
		div(\rho\Phi)=\rho{\Psi^t},
	\end{equation}  
	{(\ref{peaeq})} is equivalent to: 
	\begin{enumerate}
		\item The orthogonality of the successive gradients of  $\{\nabla^{(m)}\mathbb{P}_{n+m}^t\}_n$ with respect to the matrix weight $\rho_{m}=\rho\Phi^{\otimes\,m}$ 
		\item The partial differential equation for which the closed form is 
		\[div\left((\rho_{m}\otimes\Phi)\nabla^{(m+1)}\mathbb{P}_{n+m}^t\right)+\rho_{m}\nabla^{(m)}\mathbb{P}_{n+m}^t\Lambda_{n+m,m}=0\]
		\item The  Rodrigues formula
		\begin{eqnarray}\label{e5}
			\mathbb{P}^t_n=\frac{(-1)^n}{\rho}div^{(n)}\left[\rho\,\varPhi^{\otimes\,n}\right]R_n,
		\end{eqnarray}
		where $R_n=\left(\nabla^{(n)}\mathbb{P}_n^t\right)\prod_{j=0}^{n-1}\varLambda_{n,j}^{-1}\in \mathcal{M}_{2^n,n+1}(\mathbb{R})$.
		\item The structure relation
		\begin{align*}
			\lefteqn{\left(\Phi\otimes I_{2^m}\right)\nabla^{(m+1)}\mathbb{P}^t_{n+m}}&\\
			&=\left(I_2\otimes \nabla^{(m)}\mathbb{P}^t_{n+m+1}\right)A_{n+1}^{n,m}+\left(I_2\otimes\nabla^{(m)}\mathbb{P}^t_{n+m}\right)A_{n}^{n,m}+\left(I_2\otimes\nabla^{(m)}\mathbb{P}^t_{n+m-1}\right)A_{n-1}^{n,m},
		\end{align*}
	\end{enumerate}

	where $\varLambda_{n,j}\in \mathcal{M}_{n+1}(\mathbb{R})$, space of $(n+1,\,n+1)$ matrices with coefficients in $\mathbb{R}$, $div^{(n)}$ is the successive $n^{th}$ divergence  $div^{(n)}A=div(div^{(n-1)}A)$  and $\nabla^{(n)}$ is the successive $n^{th}$ gradients $\nabla^{(n)}B=\nabla(\nabla^{(n-1)}B)$; $\varPhi$ is a 2 lines 2 columns, positive definite matrix polynomial of degree at most 2;  $\varPhi^{\otimes\,n}$ is the  $n^{th}$ first kind Kronecker product (i.e., the tensor product), $\varPhi^{\otimes\,n}=\Phi\otimes \varPhi^{\otimes\,n-1}$, $\varPhi^{\otimes\,1}=\varPhi$.

	\subsection{Notations}
	Let us recall  some fundamental notations and results. $\mathcal{P}=\mathbb{R}[x,y]$ is the space of polynomials in two variables with real coefficients. $\mathcal{P}_n$ is the subspace of polynomials of degree (also called total degree ) less or equal to $n$. $\Pi_n$ is the subspace of polynomials of degree $n$. Let $\mathbb{N}$ be the set of non-negative integers. For $\alpha=(\alpha_1,\;\alpha_2)\in \mathbb{N}^2$, $|\alpha|=\alpha_1+\alpha_2$. Let $p\in\mathcal{P}$ be a polynomial of degree $n$,
	\begin{eqnarray}
		\nonumber p(x,y)&=&\sum_{|\alpha|\leq n}c_{\alpha}x^{\alpha_1}y^{\alpha_2},\;\;c_{\alpha}\in\mathbb{R},\\
		&=&\sum_{k=0}^{n}C_{\alpha}X_k,\label{e6}
	\end{eqnarray}
	where $X_k$ is the column vector $X_0=1$, $X_k=(x^k,x^{k-1}y,\dots,xy^{k-1},\,y^k)^t$, 
	$k\geq 1$, of size $(k+1, 1)$ and $C_k\in \mathcal{M}_{1,\,k+1}(\mathbb{R})$. $\mathcal{M}_{h,\, k}(\mathbb{R})$ and $\mathcal{M}_{h,\, k}(\mathcal{P})$ denote the linear spaces of $h$ line(s) and $k$ column(s) real and polynomial matrices respectively. When $h=k$, the second index is omitted. The degree of matrix polynomial $\varPhi\in\mathcal{M}_{h,\, k}(\mathcal{P})$ is
	\begin{equation*}
		{\rm degree}\, {\varPhi}=max\left\{{\rm degree}\, \phi_{ij},1\leq i\leq h,\; 1\leq j\leq k\right\},
	\end{equation*}
	where $\phi_{ij}$ denotes $(i,j)$-entry of $\varPhi$.
	A polynomial system is a sequence of vectors  $\left\{\mathbb{P}_n\right \}_{n\geq 0}$ of increasing size such that \cite{Kowalski,Yuan,Xu}
	\begin{equation*}\mathbb{P}_n=\left(P_{n,0},\;P_{n-1,1},\dots,P_{1,n-1},\;P_{0,n}\right)^t\in \mathcal{M}_{n+1,1}\left(\mathcal{P}_n\right),\end{equation*}
	where $P_{i,j}$, $0\leq i\leq n$,  $0\leq j\leq n$, are polynomials of degree $n$ independent modulo $\mathcal{P}_{n-1}$. Expanding  $P_{i,j}$ in the basis $X_k$ and using (\ref{e6}), $\mathbb{P}_n(x,y)$ can be written as follows 
	\begin{equation*}
		\mathbb{P}_n(x,y)=\sum_{k=0}^{n}G_{n,k}X_k,
	\end{equation*}
	where $G_{n,k}\in \mathcal{M}_{n+1,\,k+1}\left(\mathbb{R}\right)$ and $G_{n,n}$ is the leading coefficient. $\mathbb{P}_n$ is monic if $G_{n,n}$ is equal to the matrix identity, $I_{n+1}$, of $\mathcal{M}_{n+1}\left(\mathbb{R}\right)$.\newline
	Let  $\Omega$ be a domain of $\mathbb{R}^2$ and $\rho$ a weight function on $\Omega$ (i.e. a non-negative and integrable function on $\Omega$). The system of polynomials $\left\{\mathbb{P}_n\right\}_n$ is orthogonal with respect to $\rho(x,y)$ if for $n\geq 0$
	\begin{equation*}
		\begin{cases}
			\int_{\Omega}X_m\mathbb{P}^t_n(x,y)\rho(x,y)dxdy&=0,\;\;m<n,\\
			\int_{\Omega}X_n\mathbb{P}^t_n(x,y)\rho(x,y)dxdy&=S_n,
		\end{cases}
	\end{equation*}
	where $S_n$ is  an invertible matrix of size $(n+1,\,n+1)$. For $n\geq 0$, the multiplication of $X_n$ by $x$ and $y$ are respectively given by
	\begin{equation}\label{mbxy}
		xX_n=L_{n,1}X_{n+1}\;\; {\rm and }\;\; yX_n=L_{n,2}X_{n+1},
	\end{equation}
	where $L_{n,1}$ and $L_{n,2}$ are the $(n+1,\,n+2)$ matrices 
	\begin{equation}\label{Ln}
		L_{n,1}=\left( \begin{matrix}
			1&0&\cdots&0&0 \\
			0&1&&\vdots&0\\
			\vdots&&\ddots &0&\vdots\\
			0&\cdots&0&1&0
		\end{matrix}\right)\hspace{1cm}{\rm and}\hspace{1cm}L_{n,2}=\left( \begin{matrix}
			0 &1&0&\cdots&0\\
			0&0&1&&\vdots\\
			\vdots&\vdots&&\ddots&0\\
			0&0&\cdots&0&1
		\end{matrix}\right).
	\end{equation}
	On the other hand, the partial derivative of $X_n$ with respect to $x$ and $y$ are defined by \cite{Marcellan}
	\begin{equation}\label{pd}
		\partial_xX_n=N_{n,1}^tX_{n-1}, \hspace{2cm} \partial_yX_n=N_{n,2}^tX_{n-1},
	\end{equation}
	where $N_{n,k}$, $k=1,2$ are $(n,\,n+1)$ matrices defined as follow{s} $N_{0,1}=N_{0,2}=0$, for $n\geq 1$
	\begin{equation}\label{Nn}
		N_{n,1}=\left( \begin{matrix}
			n&0&\cdots&0&0 \\
			0&n-1&&\vdots&0\\
			\vdots&&\ddots &0&\vdots\\
			0&\cdots&0&1&0
		\end{matrix}\right) \hspace{1cm}{\rm and}\hspace{1cm}N_{n,2}=\left( \begin{matrix}
			0 &1&0&\cdots&0\\
			0&0&2&&\vdots\\
			\vdots&\vdots&&\ddots&0\\
			0&0&\cdots&0&n
		\end{matrix}\right).
	\end{equation}
	
	$A=\left(a_{ij}\right)_{i,j=1}^m\in \mathcal{M}_m\left(P\right)$ and $B=\left(b_{ij}\right)_{i,j=1}^n\in\mathcal{M}_n\left(P\right)$. The Kronecker product of $A$ and $B$ denoted $A\otimes B$ is the block matrix \cite[p.228]{Bellman} 
	\begin{equation}\label{cprod}
		A\otimes B=\left(\begin{matrix}
			a_{11}B&\cdots&a_{1n}B\\
			\vdots&&\vdots\\
			a_{n1}B&\cdots&{a_{nn}B}
		\end{matrix}\right).
	\end{equation}
	For a positive integer $k\geq 1$, the  $k^{th}$ Kronecker product of a matrix $A\in \mathcal{M}_m\left(\mathcal{P}\right)$ is the matrix of size $m^k$ written as 
	\begin{equation*}
		A^{\otimes\,k}=A\otimes A^{\otimes\,k-1},\;\; A^{\otimes\,0}=1.
	\end{equation*}
	Moreover, if $A$ and $C$ are matrices of the same size, $B$ and $D$ are matrices of the same size, then \cite[p.228]{Bellman}
	\begin{equation}\label{p1e2}
		\left(AC\right)\otimes\left(BD\right)=(A\otimes\,B)(C\otimes\,D).
	\end{equation}

	Let $A,\,B\in\mathcal{M}_{m,n}(\mathcal{P})$. The gradient of $A$ and  the divergence  of $\left(\begin{matrix}
		A\\
		B
	\end{matrix}\right)$  are given respectively by
	$\nabla A=\left(\begin{matrix}
		\partial_x A\\
		\partial_y A
	\end{matrix}\right)$, and 
	$div\left(\begin{matrix}
		A\\
		B
	\end{matrix}\right)=\partial_xA+\partial_yB.$ The scalar product of $\left(\begin{matrix}
		A\\
		B
	\end{matrix}\right)$ by a vector $\overrightarrow{n}\left(\begin{matrix}
		n_x\\
		n_y
	\end{matrix}\right)$ is
	\begin{equation*}
		\left(\begin{matrix}
			A\\
			B
		\end{matrix}\right)\cdot\overrightarrow{n}=An_x+Bn_y.
	\end{equation*}
	The norm of a matrix $A=(a_{ij})_{i=1,j=1}^{m,n}$ is given by $\|A\|_{max}=max_{i,j}|a_{ij}|$.
	Let be $\Omega$ a subset of $\mathbb{R}^2$. $1_{\Omega}$ is the indicator function of $\Omega$. If the boundary of $\Omega$ exists, we denote it by   $\partial\Omega$. $B(O, \varepsilon),\,\varepsilon>0$, is the open ball of $\mathbb{R}^2$ centred at $O$ with radius $\varepsilon$.
	
	The structure of the paper is as follows. In Section 2, we prove several preliminary results, Section 3 is devoted to  our main contribution, the characterization theorem, and proposes a definition of classical orthogonal polynomials in two {continuous} variables {as bivariate polynomials that are orthogonal on a simply connected open subset of $\mathbb{R}^2$ with respect to a weight function $\rho$ satisfying the Pearson equation (\ref{peaeq}) under the boundary condition {(\ref{bc0a})}, the differential system (\ref{dq0}) and the  linear {independence} of polynomial components of the vector $\Psi$}. Section 4 connects the definition of bivariate classical orthogonal polynomials based on a weight with the one based on moment functionals, and provides examples. Moreover, this section connects our Rodrigues formula with that of (\ref{e4b}).
	
	\section{Fundamental results}
	
	This subsection contains results that will be used for the proof of the main theorem of this work.
	
	\begin{lemma}
		Let $m \geq 0$ be a fixed integer and $\left\{I_{2^m}\otimes X^t_n,\;n\in \mathbb{N}\right\}$ a family of matrix vector polynomials.
		\begin{enumerate}
			\item For all $n\geq\,0$, the following algebraic properties are satisfied
			{\small	\begin{subequations}
					\begin{empheq}{align}
						& x\left(I_{2^m}\otimes X^t_n\right)=\left(I_{2^m}\otimes X^t_{n+1}\right)\left(I_{2^m}\otimes L_{n,1}^t\right),\;\;y\left(I_{2^m}\otimes X^t_n\right)=\left(I_{2^m}\otimes X^t_{n+1}\right)\left(I_{2^m}\otimes L_{n,2}^t\right)\label{e12a},\\
						&x^2\left(I_{2^m}\otimes X^t_n\right)=\left(I_{2^m}\otimes X^t_{n+2}\right)\left(I_{2^m}\otimes L_{n+1,1}^t L_{n,1}^t\right)\label{e12b},\\
						& xy\left(I_{2^m}\otimes X^t_n\right)=\left(I_{2^m}\otimes X^t_{n+2}\right)\left(I_{2^m}\otimes L_{n+1,1}^t L_{n,2}^t\right)\label{e12c},\\
						&y^2\left(I_{2^m}\otimes X^t_n\right)=\left(I_{2^m}\otimes X^t_{n+2}\right)\left(I_{2^m}\otimes L_{n+1,2}^t L_{n,2}^t\right)\label{e12d},\\
						&\left(I_{2^m}\otimes X^t\right)A\left(I_{2^m}\otimes X^t_n\right)=\left(I_{2^m}\otimes X^t_{n+1}\right)\left(I_{2^m}\otimes L_n^t\right)\left(A\otimes\,I_{n+1}\right),\; A\in \mathcal{M}_{2^{m+1},2^m}\left(\mathbb{R}\right)\label{e12e},
					\end{empheq}
				\end{subequations}
			}
			\item For all $n\geq\,1$, the following  partial differentiation properties are satisfied
			{\small
				\begin{subequations}\label{e13}
					\begin{empheq}{align}
						& \partial_x\left(I_{2^m}\otimes X^t_n\right)=\left(I_{2^m}\otimes X^t_{n-1}\right)\left(I_{2^m}\otimes N_{n,1}\right),\; \partial_y\left(I_{2^m}\otimes X^t_n\right)=\left(I_{2^m}\otimes X^t_{n-1}\right)\left(I_{2^m}\otimes N_{n,2}\right),\label{e13a}\\
						& \partial_x^2\left(I_{2^m}\otimes X^t_n\right)=\left(I_{2^m}\otimes X^t_{n-2}\right)\left(I_{2^m}\otimes N_{n-1,1}N_{n,1}\right),\label{e13b}\\
						& \partial_x\partial_y\left(I_{2^m}\otimes X^t_n\right)=\left(I_{2^m}\otimes X^t_{n-2}\right)\left(I_{2^m}\otimes N_{n-1,1}N_{n,2}\right),\label{e13c}\\
						& \partial_y^2\left(I_{2^m}\otimes X^t_n\right)=\left(I_{2^m}\otimes X^t_{n-2}\right)\left(I_{2^m}\otimes N_{n-1,2}N_{n,2}\right),\label{e13d}
					\end{empheq}
			\end{subequations}}
			where
			\begin{equation*}
				L_{n}=\left(\begin{matrix}
					L_{n,1}\\
					L_{n,2}
				\end{matrix}\right),\;\;N_{n}=\left(\begin{matrix}
					N_{n,1}\\
					N_{n,2}
				\end{matrix}\right)\;{\rm and}\;A\in\mathcal{M}_{2^{m+1},2^m}\left(\mathbb{R}\right).
			\end{equation*}
		\end{enumerate}
	\end{lemma}
	
	\begin{proof}
		
		Since  
		\begin{equation}\label{e13e}\left(I_{2^m}\otimes\,X^t_n\right)=\left( \begin{matrix}
				X_n^t&0&\cdots&0 \\
				0&X_n^t&&\vdots\\
				\vdots&&\ddots &0\\
				0&\cdots&0&X_n^t
			\end{matrix}\right)\end{equation} using (\ref{mbxy}), we have 
		\begin{eqnarray*}
			x\left(I_{2^m}\otimes\,X^t_n\right)&=&\left( \begin{matrix}
				X_{n+1}^tL_{n,1}&0&\cdots&0 \\
				0&X_{n+1}^tL_{n,1}&&\vdots\\
				\vdots&&\ddots &0\\
				0&\cdots&0&X_{n+1}^tL_{n,1}
			\end{matrix}\right)\\
			&=&\left( \begin{matrix}
				X_{n+1}^t&0&\cdots&0 \\
				0&X_{n+1}^t&&\vdots\\
				\vdots&&\ddots &0\\
				0&\cdots&0&X_{n+1}^t
			\end{matrix}\right)\left( \begin{matrix}
				L_{n,1}&0&\cdots&0 \\
				0&{L_{n,1}}&&\vdots\\
				\vdots&&\ddots &0\\
				0&\cdots&0&L_{n,1}
			\end{matrix}\right)\\
			&=&\left(I_{2^m}\otimes\,X_{n+1}^t\right)\left(I_{2^m}\otimes\,L_{n,1}\right).
		\end{eqnarray*}	
		In a similar way, we obtain $y\left(I_{2^m}\otimes\,X^t_n\right)$ as well as  (\ref{e12b})-(\ref{e12d}). Let us prove (\ref{e12e}). Let $A\in \mathcal{M}_{2^{m+1},2^m}\left(\mathbb{R}\right)$. $A=\left(a_{ij}\right)_{i,j=1}^{2^m}$, with $a_{ij}=\left(a_{ij}^{(1)},\;a_{ij}^{(2)}\right)^t$
		\begin{eqnarray*}
			\left(I_{2^m}\otimes X^t\right)A\left(I_{2^m}\otimes X^t_n\right)={\left(X^ta_{ij}X_n^t\right)_{i,j=1}^{2^m}}=\left(a_{ij}^{(1)}xX_n^t+a_{ij}^{(2)}yX_n^t\right)_{i,j=1}^{2^m}
		\end{eqnarray*}
		Taking into account (\ref{mbxy}), we obtain
		\begin{eqnarray*}
			\left(I_{2^m}\otimes X^t\right)A\left(I_{2^m}\otimes X^t_n\right)&=&\left(X_{n+1}^t\left(a_{ij}^{(1)}L_{n,1}^t+a_{ij}^{(2)}L_{n,2}^t\right)\right)_{i,j=1}^{2^m}\\
			&=&\left(I_{2^m}\otimes\,X_{n+1}^t\right)\left(a_{ij}^{(1)}L_{n,1}^t+a_{ij}^{(2)}L_{n,2}^t\right)_{i,j=1}^{2^m}
		\end{eqnarray*}
		Observing  that 
		$a_{ij}^{(1)}L_{n,1}^t+a_{ij}^{(2)}L_{n,2}^t=L_n^t\left(\begin{matrix}
			a_{ij}^{(1)}I_{n+1}\\a_{ij}^{(2)}I_{n+1}
		\end{matrix}\right)$, we obtain
		\begin{eqnarray*}
			\left(I_{2^m}\otimes X^t\right)A\left(I_{2^m}\otimes X^t_n\right)&=&\left(I_{2^m}\otimes\,X_{n+1}^t\right)\left(I_{2^m}\otimes\,L_n^t\right)\left(\begin{matrix}
				a_{ij}^{(1)}I_{n+1}\\a_{ij}^{(2)}I_{n+1}
			\end{matrix}\right)_{i,j=1}^{2^m}\\
			&=&\left(I_{2^m}\otimes\,X_{n+1}^t\right)\left(I_{2^m}\otimes\,L_n^t\right)\left(A\otimes\,I_{n+1}\right).
		\end{eqnarray*}
		
		To prove (\ref{e13a}), apply $\partial_x$ from the left of (\ref{e13e}) and use the transposed version of (\ref{pd}) to obtain the result. (\ref{e13b})-(\ref{e13d}) are obtained in a similar way. 
	\end{proof}

	\begin{lemma}\label{prop2}\hspace{1cm}\\
		\begin{enumerate}
			\item Let A and B be two matrices, the following product rule is satisfied
			\begin{equation}\label{p1e1a}
				\partial_x (A\otimes B)=\partial_xA\otimes B+A\otimes\partial_xB.
			\end{equation}
			\item For all non-negative  integers $n$ and $m$, $0\leq m\leq n$, the polynomial $Q_{n,m}=	{\nabla^{(m)}\mathbb{P}^t_{n+m}}$ of total degree $n$ can be expanded in the system  as $\left\{\left(I_{2^m}\otimes X^t_k\right)\right\}_k$
			\begin{equation*}
				Q_{n,m}(x,\,y)=\sum_{k=0}^{n}\left(I_{2^m}\otimes X^t_k\right)G_{n,m,k},
			\end{equation*}
			where the leading coefficient of $\mathbb{Q}_{n,m}$ is the $(2^m(n+1), n+m+1)$-matrix given by the recurrence relation
			\begin{equation}\label{e8a}G_{n,\,m,\,n}=\left(\begin{matrix}
					I_{2^{m-1}}\otimes\,N_{n+1,1}&0\\
					0&I_{2^{m-1}}\otimes\,N_{n+1,2}
				\end{matrix}\right)\left(\begin{matrix}
					G_{n+1,\,m-1,\,n+1}\\G_{n+1,\,m-1,\,n+1}
				\end{matrix}\right),m\geq\,1,\;\; G_{n,0,n}=I_{n+1}.\end{equation}
		\end{enumerate}
		
	\end{lemma}
	\begin{proof}
		Write $A\otimes\,B$ in the form (\ref{cprod}) and use the usual product rule to obtain (\ref{p1e1a}). Let us prove the second item by induction on $m$, $m\geq\,0$. The item is obviously satisfied for $m=0$.  Let us assume it is satisfied up to $m-1$. Therefore
		\begin{equation*}
			\mathbb{Q}_{n,m}(x,y)=\nabla\left(\nabla^{(m-1)}\mathbb{P}_{n+1+m-1}^t\right)(x,\,y)=\nabla\left(\sum_{k=0}^{n+1}\left(I_{2^{m-1}}\otimes\,X_{k}^t\right)G_{n+1,m-1,k}\right).
		\end{equation*}
		Next, we use (\ref{e13a}) and rewrite the obtained sums from 0 to get
		\begin{eqnarray*}
			\mathbb{Q}_{n,m}(x,y)&=&\left(\begin{matrix}
				\sum_{k=0}^{n}\left(I_{2^{m-1}}\otimes\,X_{k}^t\right)\left(I_{2^{m-1}}\otimes\,N_{k+1,1}\right)G_{n+1,m-1,k+1}\\
				\\
				\sum_{k=0}^{n}\left(I_{2^{m-1}}\otimes\,X_{k}^t\right)\left(I_{2^{m-1}}\otimes\,N_{k+1,2}\right)G_{n+1,m-1,k+1}\end{matrix}\right)\\
			&=&\sum_{k=0}^{n}\left(\begin{matrix}
				I_{2^{m-1}}\otimes\,X_{k}^t&0\\
				0&I_{2^{m-1}}\otimes\,X_{k}^t
			\end{matrix}\right)\left(\begin{matrix}
				I_{2^{m-1}}\otimes\,N_{k+1,1}&0\\
				0&I_{2^{m-1}}\otimes\,N_{k+1,2}
			\end{matrix}\right)\left(\begin{matrix}
				G_{n+1,m-1,k+1}\\
				G_{n+1,m-1,k+1}
			\end{matrix}\right)\\
			&=&\sum_{k=0}^{n}\left(I_{2^{m}}\otimes\,X_k^t\right)G_{n,m,k}.
		\end{eqnarray*}
	\end{proof}
	
	\begin{proposition}\label{p2a} Let $\Omega$ be a simply connected open subset of $\mathbb{R}^2$ and $\rho$ a weight function on $\Omega$. Let be $A\in\mathcal{M}_2\left(\mathcal{P}\right)$, $\Omega_j=\Omega\cap B(O,j),\,j=1,2\dots,$ and $\overrightarrow{n_j}$ an outer vector normal of $\partial\Omega_j$.
		If the Neumann type boundary condition 
		\begin{equation}\label{Neumannc}
			\lim\limits_{j\rightarrow\infty}1_{\partial\Omega_j}\left(\rho A\,\nabla\,u\right)\cdot\overrightarrow{n_j}=0,\,\,\forall u\in \mathcal{P}
		\end{equation}
		is satisfied, then
		\begin{equation}\label{bc0}
			\int_{\Omega}div\left(\rho\left(I_2\otimes\,M\right)A^{\otimes\,m}\nabla\,N\right)dxdy=0,
		\end{equation}
		for all $M\in\mathcal{M}_{p,2^{m-1}}\left(\mathcal{P}\right)$ and $N\in\mathcal{M}_{2^{m-1},r}\left(\mathcal{P}\right)$, where $m$, $p$ and $r$ are positive integers.
	\end{proposition}
	\begin{proof}
		First prove that $\lim\limits_{j\rightarrow\infty}1_{\partial\Omega_j}\left(\rho\left(I_2\otimes\,M\right)A^{\otimes\,m}\nabla N \right)\cdot \overrightarrow{n_j}=0$ then use the divergence theorem as well as Lebesgue's dominated convergence theorem to obtain (\ref{bc0}). For this purpose, 
		observe that 
		\begin{equation*}
			\rho\left(\left(I_2\otimes\,M\right)A^{\otimes\,m}\nabla N \right)\cdot \overrightarrow{n}=\rho\left(a_{11}n_x+a_{21}n_y\right)MA^{\otimes\,m-1}\partial_x\,N+\rho\left(a_{12}n_x+a_{22}n_y\right)MA^{\otimes\,m-1}\partial_y\,N
		\end{equation*}
		and take $MA^{\otimes\,m-1}\partial_x\,N=(p_{ik})_{i=1,k=1}^{p,r}$, $MA^{\otimes\,m-1}\partial_y\,N=(q_{ls})_{l=1,s=1}^{p,r}$ , where $p_{ik}$, $q_{ls}\in \mathcal{P}$, to obtain
		\begin{eqnarray*}
			\left(\left(I_2\otimes\,M\right)A^{\otimes\,m}\nabla N \right)\cdot \overrightarrow{n}&=&\left(\rho\left(a_{11}n_x+a_{21}n_y\right)p_{ik}+\rho\left(a_{12}n_x+a_{22}n_y\right)q_{ik}\right)_{i=1,\,k=1}^{p,\,r}\\
			&=&\left(\left(\left(\rho\,A\right)\cdot\overrightarrow{n}\right)\left(\begin{matrix}
				p_{ik}\\
				q_{ik}
			\end{matrix}\right)\right)_{i=1,\,k=1}^{p,\,r}.
		\end{eqnarray*}
		Since $\left\{\nabla \mathbb{P}^t_n\right\}_{n\geq 1}$ is a basis of $\mathcal{P}\times\mathcal{P}$, for fixed $i$ and $k$,  
		\begin{equation*}
			\left( \begin{matrix}
				p_{ik}\\
				q_{ik}
			\end{matrix}\right)={\sum_{l=0}^{z+1}}\nabla\mathbb{P}_l^tC_{n,l}^{i,\,k},\;\;z={\rm degree} \left( \begin{matrix}
				p_{ik}\\
				q_{ik}
			\end{matrix}\right).
		\end{equation*}
		Therefore, 
		\begin{eqnarray*}
			\left(\rho\left(I_2\otimes\,M\right)A^{\otimes\,m}\nabla N \right)\cdot \overrightarrow{n}
			&=&\left({\sum_{l=0}^{z+1}}\left(\left(\rho\,A\right)\cdot\overrightarrow{n}\right)\nabla\mathbb{P}_l^tC_{n,l}^{i,\,k}\right)_{i=1,\,k=1}^{p,\,r}.
		\end{eqnarray*}
		Replace $\overrightarrow{n}$  by $\overrightarrow{n}_j$,  multiply from the left by the indicator function of $\partial\Omega_j$, $1_{\partial\Omega_j}$, and use the fact that $\left(\left(\rho\,A\right)\cdot\overrightarrow{n_j}\right)\nabla\mathbb{P}_l^t= \left(\rho\,A\nabla\mathbb{P}_l^t\right)\cdot\overrightarrow{n_j}$ as well as the hypothesis (\ref{Neumannc}) to get
		\begin{equation}\label{mnbc}\lim\limits_{j\rightarrow\infty}1_{\partial\Omega_j}\left(\rho\left(I_2\otimes\,M\right)A^{\otimes\,m}\nabla N \right)\cdot \overrightarrow{n_j}=0.\end{equation}
		To end this proof, observe that 
		$$\int_{\Omega}div\left(\rho\left(I_2\otimes\,M\right)A^{\otimes\,m}\nabla\,N\right)dxdy=\lim\limits_{j\rightarrow\infty}\int_{{\Omega}_j}div\left(\rho\left(I_2\otimes\,M\right)A^{\otimes\,m}\nabla\,N\right)dxdy.$$
		Since $\Omega_j$ is a connected open bounded subset of $\mathbb{R}^2$ it has a piecewise smooth boundary. So, from the divergence theorem,
		$\int_{{\Omega}_j}div\left(\rho\left(I_2\otimes\,M\right)A^{\otimes\,m}\nabla\,N\right)dxdy=\int_{\partial{\Omega}_j}\left(\rho\left(I_2\otimes\,M\right)A^{\otimes\,m}\nabla\,N\right)\cdot \overrightarrow{n_j}dxdy.$
		Therefore
		\[ \int_{\Omega}div\left(\rho\left(I_2\otimes\,M\right)A^{\otimes\,m}\nabla\,N\right)dxdy=\lim\limits_{j\rightarrow\infty}\int1_{\partial{\Omega}_j}\left(\rho\left(I_2\otimes\,M\right)A^{\otimes\,m}\nabla\,N\right)\cdot \overrightarrow{n_j}dxdy\]
		and $\left\|\int_{\Omega}div\left(\rho\left(I_2\otimes\,M\right)A^{\otimes\,m}\nabla\,N\right)dxdy\right\|_{max}\leq\lim\limits_{j\rightarrow\infty}\int\left\|1_{\partial{\Omega}_j}\left(\rho\left(I_2\otimes\,M\right)A^{\otimes\,m}\nabla\,N\right)\cdot \overrightarrow{n_j}\right\|_{max}dxdy$. From (\ref{mnbc}),
		$\lim\limits_{j\rightarrow\infty}\left\|1_{\partial{\Omega}_j}\left(\rho\left(I_2\otimes\,M\right)A^{\otimes\,m}\nabla\,N\right)\cdot \overrightarrow{n_j}\right\|_{max}=0$. Therefore there exists $j_0\geq\,1$ such that \[\left\|1_{\partial{\Omega}_j}\left(\rho\left(I_2\otimes\,M\right)A^{\otimes\,m}\nabla\,N\right)\cdot \overrightarrow{n_j}\right\|_{max}\leq \left\|1_{\partial{\Omega}_{j_0}}\left(\rho\left(I_2\otimes\,M\right)A^{\otimes\,m}\nabla\,N\right)\cdot \overrightarrow{n_{j_{0}}}\right\|_{max}, j\geq j_0.\] Use Lebesgue's dominated convergence theorem to have
		\small{\[\lim\limits_{j\rightarrow\infty}\int\left\|1_{\partial{\Omega}_j}\left(\rho\left(I_2\otimes\,M\right)A^{\otimes\,m}\nabla\,N\right)\cdot \overrightarrow{n_j}\right\|_{max}dxdy=\int\lim\limits_{j\rightarrow\infty}\left\|1_{\partial{\Omega}_j}\left(\rho\left(I_2\otimes\,M\right)A^{\otimes\,m}\nabla\,N\right)\cdot \overrightarrow{n_j}\right\|_{max}dxdy=0.\]} Thus $\left\|\int_{\Omega}div\left(\rho\left(I_2\otimes\,M\right)A^{\otimes\,m}\nabla\,N\right)dxdy\right\|_{max}=0.$
	\end{proof}

	\begin{proposition}\label{p3} Let $A$ be a $(2,2)$-matrix function.\\
		\begin{enumerate}
			\item 	Consider $M\in\mathcal{M}_{p,2^{m-1}}\left(\mathcal{P}\right)$ and $N\in\mathcal{M}_{2^{(m-1)},q}\left(\mathcal{P}\right)$. The following relation holds
			\begin{equation}\label{p2e1}
				div\left[\left(I_2\otimes\,M\right)A^{\otimes\,m}\nabla\,N\right]=Mdiv(A^{\otimes\,m}\nabla N)+\left(\partial_xM,\partial_yM\right)\left(A^{\otimes\,m}\nabla N\right).
			\end{equation} 
			
			\item Let $m$ be a positive integer,{$B\in\mathcal{M}_2\left(\mathcal{P}\right)$} and $N\in \mathcal{M}_{2^{m},q}\left(\mathcal{P}\right)$. The following relation holds
			\begin{equation}\label{p2e2}
				div\left[\left( B\otimes\,A^{\otimes m}\right)\nabla\,N\right]=A^{\otimes\,m}\left[\left(\left(B\nabla\right)\cdot\nabla\right)N\right]+div\left(B\otimes\,A^{\otimes\,m}\right)\nabla\,N.
			\end{equation}
		\end{enumerate}
	\end{proposition}
	
	\begin{proof}
		To prove the first item, write
		\begin{equation*}
			\left(I_2\otimes\,M\right)A^{\otimes\,m}\nabla N=\left(\begin{matrix}
				a_{11}MA^{\otimes\,m-1}\partial_x N+a_{12}MA^{\otimes\,m-1}\partial_y N\\
				a_{21}MA^{\otimes\,m-1}\partial_x N+a_{22}MA^{\otimes\,m-1}\partial_y N
			\end{matrix}\right),\, A=\left(\begin{matrix}
				a_{11}&a_{12}\\
				a_{21}&a_{22}
			\end{matrix}\right),
		\end{equation*}
		and apply $div$ on both sides to obtain the result. As for the last item of this proposition, apply $div$ on both sides of the relation 
		\begin{equation*}
			\left(B\otimes\,A^{\otimes\,m}\right)=\left(\begin{matrix}
				A^{\otimes\,m}\left(b_{11}\partial_x\,N+b_{12}\partial_y\,N\right)\\
				A^{\otimes\,m}\left(b_{21}\partial_x\,N+b_{22}\partial_y\,N\right)
			\end{matrix}\right),\;\, B=\left(\begin{matrix}
				b_{11}&b_{12}\\
				b_{21}&b_{22}
			\end{matrix}\right),
		\end{equation*}
		and use the usual product rule to obtain the result. The first item can also be obtained  by replacing  $\phi\otimes I_{2^{m-1}}$ by $A^{\otimes m}$ into the second item of \cite[Lemma 3.1]{Fern}. 
	\end{proof}
	\begin{proposition}\label{prop4}
		Let $m$ and $n$ be two non-negative integers, $0\leq m\leq n$. Let $\Phi$ be a {$2$}-matrix polynomial of total degree at most two. Let $\psi_i^{(m)}(x,y)=\left(I_{2^m}\otimes X^t\right)D_i^{(m)}+E_i^{(m)}$ with $D_i^{(m)}\in\mathcal{M}_{2^{m+1},\,2^m}\left(\mathbb{R}\right)$ and $E_i^{(m)}\in\mathcal{M}_{2^m}\left(\mathbb{R}\right)$, $i=1,2$ be two $2^m$-matrix polynomials of total degree one. Then
		\begin{equation}\label{e9}
			\left(\left(\Phi\nabla\right)\cdot\nabla\right)\mathbb{Q}_{n,m}+\psi_1^{(m)}\partial_x\mathbb{Q}_{n,m}+\psi_2^{(m)}\partial_y\mathbb{Q}_{n,m}=\left(I_{2^m}\otimes\,X_n^t\right)T_n^{(m)}G_{n,m,n}+\cdots,
		\end{equation}
		where
		\begin{equation*}
			T_n^{(m)}=\left(L_{n-1}^{(m)*}\right)^t\left(\left(A_1,2A_2,A_3\right)\otimes\,I_{2^m(n-1)}\right)N_n^{(m)*}+\left(I_{2^m}\otimes L_{n-1}^t\right)^t\left(\left(D_1^{(m)},D_2^{(m)}\right)\otimes\,I_n\right)N_n^{(m)}
		\end{equation*}
		with 
		$A_i=\left(a_{i1},a_{i2},a_{i3}\right)^t$, $i=1,2,3$, $D_j=\left(d_{j1},d_{j2}\right)^t$, $j=1,2$,
		\begin{eqnarray*}
			L_{n}^{(m)*}&=&
			\left(\begin{matrix}
				I_{2^m}\otimes\,L_{n-1,1}L_{n,1}\\
				I_{2^m}\otimes\,L_{n-1,2}L_{n,1}\\
				I_{2^m}\otimes\,L_{n-1,2}L_{n,2}
			\end{matrix}\right),\;\;N_{n}^{(m)*}=
			\left(\begin{matrix}
				I_{2^m}\otimes\,N_{n-1,1}N_{n,1}\\
				I_{2^m}\otimes\,N_{n-1,2}N_{n,1}\\
				I_{2^m}\otimes\,N_{n-1,2}N_{n,2}
			\end{matrix}\right),\\
			L_{n}^{(m)}&=&
			\left(\begin{matrix}
				I_{2^m}\otimes\,L_{n,1}\\
				I_{2^m}\otimes\,L_{n,2}
			\end{matrix}\right) \;{\rm and}\; N_{n}^{(m)}=
			\left(\begin{matrix}
				I_{2^m}\otimes\,N_{n,1}\\
				I_{2^m}\otimes\,N_{n,2}
			\end{matrix}\right),
		\end{eqnarray*}
		where $L_{n,j}$ and $N_{n,j}$ are matrices in (\ref{Ln}) and (\ref{Nn}) respectively. $L_0^{(0)*}\equiv\,0$, $N_0^{(0)*}\equiv\,0$, $L_n^{(0)}\equiv\,L_n$ and $N_n^{(0)}\equiv\,N_n$.  $G_{n,m,n}$ is the $\left(2^m(n+1),n+m+1\right)$-matrix given by (\ref{e8a}) and $T_n^{(m)}$ a  $({2^m(n+1)})$-matrix with
		$T_n^{(0)}\equiv\,T_n$.
	\end{proposition}
	\begin{proof}
		
		Observe that 
		\begin{equation*}
			\left(\left(\Phi\nabla\right)\cdot\nabla\right)\,\mathbb{Q}_{n,m}+\Psi\nabla\mathbb{Q}_{n,m}=\phi_{1,1}\partial_x^2\mathbb{Q}^t_{n,m}+2\phi_{1,2}\partial_{xy}^2\mathbb{Q}_{n,m}+\phi_{2,2}\partial_y^2\mathbb{Q}_{n,m}+\psi_1^{(m)}\partial_x\mathbb{Q}_{n,m}+\psi_2^{(m)}\partial_y\mathbb{Q}_{n,m}.
		\end{equation*}
		Write  $\mathbb{Q}_{n,m}=\left(I_{2^m}\otimes\,X^t_n\right)G_{n,m,n}+\dots$ and use the partial derivatives of $\left(I_{2^m}\otimes\,X^t_{n}\right)$ given by (\ref{e13}) to obtain 
		\begin{align}\label{e8}
			&\left(\left(\Phi\nabla\right)\cdot\nabla\right)\,\mathbb{Q}_{n,m}+\psi_1^{(m)}\partial_x\mathbb{Q}_{n,m}+\psi_2^{(m)}\partial_y\mathbb{Q}_{n,m}\\
			&=\left[\phi_{1,1}\left(I_{2^m}\otimes\,X^t_{n-2}\right)\left(I_{2^m}\otimes\,N_{n-1,1}N_{n,1}\right)+2\phi_{1,2}\left(I_{2^m}\otimes\,X^t_{n-2}\right)\left(I_{2^m}\otimes\,N_{n-1,1}N_{n,2}\right)\right]G_{n,m,n}\nonumber\\
			&+\left[\phi_{2,2}\left(I_{2^m}\otimes\,X^t_{n-2}\right)\left(I_{2^m}\otimes\,N_{n-1,2}N_{n,2}\right)+\psi_1^{(m)}\left(I_{2^m}\otimes\,X^t_{n-1}\right)\left(I_{2^m}\otimes\,N_{n,1}\right)\right]G_{n,m,n}\nonumber\\
			&+\left[\psi_2^{(m)}\left(I_{2^m}\otimes\,X^t_{n-1}\right)\left(I_{2^m}\otimes\,N_{n,2}\right)\right]G_{n,m,n}+\dots\nonumber.
		\end{align}
		Since we assumed there is no restriction on the shape of the polynomials $\phi_{ij},\, i,j=1,2$ and $\psi_k,\,k=1,2$, expand $\phi_{1,1}$ as
		\begin{equation*}
			\phi_{1,1}(x,y)=a_{11}x^2+a_{12}xy+a_{13}y^2+b_{11}x+b_{12}y+c_1=X_2^tA_1+X^tB_1+c_1,
		\end{equation*}
		where $A_1=\left(a_{11},\,a_{12},\, a_{13}\right)$, $B_1=\left(b_{11},\,b_{12}\right)$ and take into account (\ref{e12a})-(\ref{e12d}) to obtain 
		\begin{align*}
			&\phi_{1,1}(x,y)\left(I_{2^m}\otimes\,X^t_{n-2}\right)\\
			&=\left(I_{2^m}\otimes\,X^t_{n}\right)\left[a_{11}\left(I_{2^m}\otimes\,L_{n-1,1}^tL_{n-2,1}^t\right)+a_{12}\left(I_{2^m}\otimes\,L_{n-1,1}^tL_{n-2,2}^t\right)+a_{13}\left(I_{2^m}\otimes\,L_{n-1,2}^tL_{n-2,2}^t\right)\right]\\
			&+\left(I_{2^m}\otimes\,X_{n-1}^t\right)\left[b_{11}\left(I_{2^m}\otimes\,L_{n-2,1}^t\right)+b_{12}\left(I_{2^m}\otimes\,L_{n-2,2}^t\right)\right]+c_1\left(I_{2^m}\otimes\,X_{n-2}^t\right)\\
			&=\left(I_{2^m}\otimes\,X^t_{n}\right)\left(L_{n-1}^{(m)*}\right)^t\left(A_1\otimes\,I_{2^m(n-1)}\right)+\left(I_{2^m}\otimes\,X^t_{n-1}\right)\left(L_{n-2}^{(m)}\right)^t\left(B_1\otimes\,I_{2^m(n-1)}\right)+c_1\left(I_{2^m}\otimes\,X^t_{n-2}\right).
		\end{align*}
		Replace $\left( A_1,\,B_1, c_1\right)$ by $\left( A_i,\,B_i, c_i\right)$, $i=2,\,3$,  to get $\phi_{1,2}(x,y)\left(I_{2^m}\otimes\,X^t_{n-2}\right)$ and $\phi_{2,2}(x,y)\left(I_{2^m}\otimes\,X^t_{n-2}\right)$ in terms of $\left(I_{2^m}\otimes\,X^t_{n}\right)$, $\left(I_{2^m}\otimes\,X^t_{n-1}\right)$ and $\left(I_{2^m}\otimes\,X^t_{n-2}\right)$. Use (\ref{e12e}) to have 
		\begin{align*}
			\psi_1^{(m)}\left(I_{2^m}\otimes\,X^t_{n-1}\right)&=\left(I_{2^m}\otimes\,X^t\right)D_1^{(m)}\left(I_{2^m}\otimes\,X^t_{n-1}\right)+E_1^{(m)}\left(I_{2^m}\otimes\,X^t_{n-1}\right)\\
			&=\left(I_{2^m}\otimes\,X^t_{n}\right)\left(I_{2^m}\otimes\,L_{n-1}\right)\left(D_1^{(m)}\otimes\,I_n\right)+E_1^{(m)}\left(I_{2^m}\otimes\,X^t_{n-1}\right).
		\end{align*}
		Replace $\left(D_1^{(m)},\,E_1^{(m)}\right)$ by $\left(D_2^{(m)},\,E_2^{(m)}\right)$ to expand $\psi_2^{(m)}(x,y)\left(I_{2^m}\otimes\,X^t_{n-1}\right)$ in terms of $\left(I_{2^m}\otimes\,X^t_{n}\right)$ and $\left(I_{2^m}\otimes\,X^t_{n-1}\right)$.
		Finally, substitute into (\ref{e8}) to obtain
		\begin{align*}
			&\left(\Phi\nabla\right)\cdot\nabla\,\mathbb{Q}_{n,m}+\psi_1^{(m)}\partial_x\mathbb{Q}_{n,m}+\psi_2^{(m)}\partial_y\mathbb{Q}_{n,m}\\
			&=\left(I_{2^m}\otimes\,X^t_{n}\right)\left(L_{n-1}^{(m)*}\right)^t\left[\left(A_{1}\otimes\,I_{2^m(n-1)}\right)\left(I_{2^m}\otimes\,N_{n-1,1}N_{n,1}\right)+2\left(A_{2}\otimes\,I_{2^m(n-1)}\right)\left(I_{2^m}\otimes\,N_{n-1,1}N_{n,2}\right)\right]\nonumber\\
			&\times\,G_{n,m,n}+\left(I_{2^m}\otimes\,X^t_{n}\right)\left(L_{n-1}^{(m)*}\right)^t\left[\left(A_{3}\otimes\,I_{2^m(n-1)}\right)\left(I_{2^m}\otimes\,N_{n-1,2}N_{n,2}\right)\right]G_{n,m,n}\nonumber\\
			&+\left(I_{2^m}\otimes\,X^t_{n}\right)\left(I_{2^m}\otimes\, L_{n-1}^t\right)\left[\left(D_1^{(m)}\otimes\,I_n\right)\left(I_{2^m}\otimes\,N_{n,1}\right)+\left(D_2^{(m)}\otimes\,I_n\right)\left(I_{2^m}\otimes\,N_{n,2}\right)\right]G_{n,m,n}+\dots\nonumber\\
			&=\left(I_{2^m}\otimes\,X^t_{n}\right)T_n^{(m)}G_{n,m,n}+\dots\,.\nonumber
		\end{align*} 
		with $T_n^{(m)}=\left(L_{n-1}^{(m)*}\right)^t\left(\left(A_1,\,2A_2,\,A_3\right)\otimes\,I_{2^m(n-1)}\right)N_n^{(m)*}+\left(I_{2^m}\otimes\, L_{n-1}^t\right)\left(\left(D_1^{(m)},D_2^{(m)}\right)\otimes\,I_n\right)N_n^{(m)}$.
	\end{proof}
	
	\begin{proposition}\label{prop3}
		Let $\left\{\mathbb{P}_n\right\}_n$ be a family of vector polynomials, orthogonal with respect to the weight function $\rho$, satisfying the Pearson equation $div\left(\rho\Phi\right)=\rho\left(\psi_1,\psi_2\right)$ with  the Neumann boundary condition (\ref{Neumannc}), where $\Phi\in\mathcal{M}_2\left(\mathcal{P}\right)$ and $\psi_i(X)=X^tD_i+E_i$ is a polynomial of total degree one, $i=1,\,2$,  with $det(D_1,\,D_2)\neq 0$  then
		\begin{equation*}
			{\rm degree}\left(\left(\left(\Phi\nabla\right)\cdot\nabla\right)\,\mathbb{P}^t_n+\Psi\nabla\mathbb{P}_n^t\right)=n,\;\;\forall n,\; n\geq 1.
		\end{equation*}
	\end{proposition}
	
	\begin{proof} 
		
		For the proof, we are going to consider two cases. 
		In the first case, degree $\Phi$ =2.\\
		Take $M=\mathbb{P}_n$, $m=1$, $A=\rho\Phi$ and $N=\mathbb{P}_j^t$ into (\ref{p2e1}) to obtain
		\begin{equation*}
			div\left[\left(I_2\otimes\mathbb{P}_n\right)\rho\Phi\nabla\mathbb{P}_j^t\right]=\mathbb{P}_ndiv\left(\rho\Phi\nabla\mathbb{P}_j^t\right)+\left(\nabla\mathbb{P}_n^t\right)^t\rho\Phi\nabla\mathbb{P}_j^t
		\end{equation*}
		
		Integrate both sides on $\Omega$ and use the boundary condition (\ref{Neumannc}), (\ref{p2e2}) with $m=0$, $B=\rho\Phi$ and $N=\mathbb{P}_j^t$ as well as the hypothesis $div\left(\rho\Phi\right)=\rho\left(\psi_1,\psi_2\right)$ to obtain
		\begin{eqnarray*}
			-\int_{\Omega}\left(\nabla\mathbb{P}_n^t\right)^t\rho\Phi\nabla\mathbb{P}_j^tdxdy=\int_{\Omega}\mathbb{P}_n\left[\left(\rho\Phi\nabla\right)\cdot\nabla \mathbb{P}_j^t+\rho\left(\psi_1,\psi_2\right)\nabla\mathbb{P}_j^t\right]dxdy.
		\end{eqnarray*}
		Take into account (\ref{e9}) (with $m=0$ and $n=j$), write $X_j^t=\mathbb{P}_j^t+\dots$ and use the orthogonality of the system $\{\mathbb{P}_n^t\}_n$ with respect to $\rho$ to obtain
		
		\begin{eqnarray*}-\int_{\Omega}\left(\nabla\mathbb{P}_n^t\right)^t\rho\Phi\nabla\mathbb{P}_j^tdxdy=\int_{\Omega}\mathbb{P}_n \mathbb{P}_n^tT_n\delta_{n,j}dxdy=S_{n}T_n\delta_{n,j},
		\end{eqnarray*}
		where $\delta_{n,j}$ is the Kronecker symbol. Therefore, multiplying from the left hand side of the expansion 
		\begin{equation*}
			\left(\begin{matrix}
				\mathbb{P}_{n+2}^t\\
				\mathbb{P}_{n+2}^t
			\end{matrix}\right)=\sum_{k=0}^{n+2}\nabla\mathbb{P}_{k+1}^tA_{k+1}
		\end{equation*}
		by $\left(\nabla \mathbb{P}_{n+1}^t\right)^t\rho\Phi$ and integrating both sides on $\Omega$, we obtain
		
		\begin{equation}\label{e10}
			\int_{\Omega}\left(\nabla \mathbb{P}_{n+1}^t\right)^t\rho\Phi\left(\begin{matrix}
				\mathbb{P}_{n+2}^t\\
				\mathbb{P}_{n+2}^t
			\end{matrix}\right)=-S_{n+1}T_{n+1}A_{n+1}.
		\end{equation}
		Since 
		\begin{eqnarray*}
			\left(\nabla \mathbb{P}_{n+1}^t\right)^t\rho\Phi\left(\begin{matrix}
				\mathbb{P}_{n+2}^t\\
				\mathbb{P}_{n+2}^t\end{matrix}\right)&=&\left(\phi_{1,1}\partial_x\mathbb{P}_{n+1}+\phi_{1,2}\partial_y\mathbb{P}_{n+1}+\phi_{2,1}\partial_x\mathbb{P}_{n+1}+\phi_{2,2}\partial_y\mathbb{P}_{n+1}\right)\mathbb{P}_{n+2}^t\\
			&=&C_nX_{n+2}\mathbb{P}^t_{n+2}+\dots,
		\end{eqnarray*}
		and $\Phi$ is of total degree $2$, $C_n$ is a $(n+2,n+3)$-matrix different from $0$. Therefore the left hand side of (\ref{e10}) is different from $0$. Thus $T_{n+1}\neq\,0$.\\
		If degree $\Phi<2$, then 
		\begin{equation*}
			T_n=\left(L_{n-1}\right)^t\left(\left(D_1,D_2\right)\otimes\,I_n\right)N_n.
		\end{equation*}
		Observing that $\left(D_1,D_2\right))\otimes\,I_n\in \mathcal{M}_{2n}(\mathbb{R})$, $N_n\in \mathcal{M}_{2n,n+1}(\mathbb{R})$ and using the Sylvester inequality, we obtain 
		\begin{eqnarray*}
			Rank[\left(D_1,D_2\right))\otimes\,I_n]+Rank(N_n)-2n\leq &Rank[\left(\left(D_1,D_2\right)\otimes\,I_n\right)\,N_n]\\
			\leq &min\left(Rank[\left(D_1,D_2\right))\otimes\,I_n],\,Rank(N_n)\right).
		\end{eqnarray*}
		Moreover, $Rank \left(\left(D_1,D_2\right)\otimes\,I_n\right)= Rank[\left(D_1,D_2\right))]Rank(I_n)$ and 
		$Rank[\left(D_1,D_2\right)]=2$ (for $det\left(D_1,D_2\right)\neq 0$). Therefore $Rank \left(\left(D_1,D_2\right)\otimes\,I_n\right)=2n$ and $Rank[\left(\left(D_1,D_2\right)\otimes\,I_n\right)\,N_n]=n+1$.
		Taking $A=\left(L_{n-1}\right)^t$ and $B=\left(\left(D_1,D_2\right)\otimes\,I_n\right)N_n$ and using again the Sylvester inequality, we have 
		\begin{eqnarray*}
			n+1= Rank(A^t)+Rank(B^t)-(n+1)\leq Rank(B^tA^t)\leq\,min(Rank(B^t),Rank(A^t))=n+1.
		\end{eqnarray*}
		So, $Rank(AB)=n+1$. i.e. $Rank(T_n)=n+1$. Therefore $T_n$ is invertible, for $T_n$ is an $(n+1)$-matrix.
		
	\end{proof}

	We now consider a Hilbert space that will be useful for the proof of the main theorem of this work. For that purpose we consider the weighted Lebesgue space
	\[ L^2(\Omega, \rho) =\left \{ u: \Omega \to \mathbb{R}\,: \text{$u$ measurable and }\; \int_\Omega u^2\rho\, dxdy < + \infty\right\}\]
	equipped with the scalar product	
	$\langle v,u\rangle_{\rho}= \int_{\Omega} uv\rho\,dxdy.$  Clearly the corresponding norm is $\|u\|_{\rho}=\left(\int_{\Omega} u^2\rho\,dxdy\right)^{1\over 2}$.
	
	\begin{theorem}\cite{NN2024}\label{NN} Let $\Omega$ be a simply connected open subset of $\mathbb{R}^2$ and $\rho$ be a weight function on $\Omega$. Let $\Phi$ be a symmetric matrix of $\mathcal{M}_{2}\left(\mathcal{P}_2\right)$, positive definite on $\Omega$ such that $\rho$ is solution to (\ref{peaeq}) under  the Neumann type boundary condition (\ref{Neumannc}). Then $\mathcal{P}$ is dense in $L^2(\Omega,\rho)$.
	\end{theorem}
	
	\section{Characterization theorem of classical orthogonal polynomials in two {continuous} variables}
	{ Using the moment functional approach as well as the second Kronecker product, \cite{FPP} proved the orthogonality relation (\ref{or1}) and the structure relation (\ref{FSR1}). The same  approach led to the Rodrigues formula (\ref{e4b}) (see \cite{Miguel2009}). 
		Our theorem below  is based on a weight and is  close to that of \cite[§5]{Al-Salam 1990} for  orthogonal polynomials in one continuous variable. It involves a boundary condition for the weight and {a tensor} product of the matrix $\Phi$. Comparison with previous works is analysed in depth in Remark \ref{GR} and Section 4.} 
	
	\begin{remark}\label{l1}
		Let $D_1=\left(d_{11},\,d_{12}\right)^t$ and $D_2=\left(d_{21},\,d_{22}\right)^t$ be two (2,1)-vectors. By induction on $m$, we obtain the determinant of the  $2m$-matrix $\left(I_m\otimes\,D_1,\;I_m\otimes\,D_2\right)$
		,$m\geq\,1$, as follows
		\begin{equation*}
			det \left(I_m\otimes\,D_1,\;I_m\otimes\,D_2\right)=\left(-1\right)^{\left[m\over 2\right]}\left[det\left(D_1,\; D_2\right)\right]^m.
		\end{equation*}
	\end{remark}
	
	\begin{proposition}\label{l2}
		Polynomials defined by
		
		\begin{subequations}
			\begin{empheq}[left=\empheqlbrace]{align}
				&\psi^{(m)}_1=I_2\otimes\psi^{(m-1)}_1+\nabla\left(\phi_{1,1},\phi_{2,1}\right)\otimes I_{2^{m-1}},\label{te4aa}\\
				&\psi^{(m)}_2=I_2\otimes\psi^{(m-1)}_2+\nabla\left(\phi_{1,2},\phi_{2,2}\right)\otimes I_{2^{m-1}}\label{te4bb}
			\end{empheq}
		\end{subequations}
		can be written as 
		\begin{equation}\label{e15}
			\psi_i^{(m)}(x,y)=\left(I_{2^m}\otimes\,X^t\right)D_i^{(m)}+E_i^{(m)},\; i=1,2,
		\end{equation}
		where
		\begin{align*}
			D_i^{(m+1)}&=H_i^{(m)}+I_2\otimes\,D_i^{(m)},\;\;D_i^{(0)}\equiv\,D_i,\\
			E_i^{(m+1)}&=K_i^{(m)}+I_2\otimes\,E_i^{(m)},\;\;E_i^{(0)}\equiv\,E_i 
		\end{align*} with 
		\begin{align*}
			H_i^{(m)}= \left(\begin{matrix}
				I_{2^{m}}\otimes\,\left(N_{2,1}A_i\right)&I_{2^{m}}\otimes\,\left(N_{2,1}A_{i+1}\right)\\
				I_{2^{m}}\otimes\,\left(N_{2,2}A_i\right)&I_{2^{m}}\otimes\,\left(N_{2,2}A_{i+1}\right)
			\end{matrix}\right),\, K_i^{(m)}=\left(\begin{matrix}
				\left(N_{1,1}B_i\right)I_{2^{m}}&\left(N_{1,1}B_{i+1}\right)I_{2^{m}}\\
				\left(N_{1,2}B_i\right)I_{2^{m}}&\left(N_{1,2}B_{i+1}\right)I_{2^{m}}
			\end{matrix}\right),
		\end{align*}
		$m\geq\,0.$ 
		
	\end{proposition}
	\begin{proof}
		We show by induction on $m$ that (\ref{e15}) is
		satisfied for $m\geq 0$. It is obvious that (\ref{e15}) is satisfied for $m=0$. Assume that they are satisfied up to a fixed integer $m> 0$.
		We use the relation (\ref{te4aa}) to obtain
		\begin{eqnarray*}
			\psi_1^{(m+1)}&=&\left(\begin{matrix}
				\partial_x\phi_{1,1}I_{2^{m}}+\psi_{1}^{(m)}&\partial_x\phi_{2,1}I_{2^{m}}\\
				\partial_y\phi_{1,1}I_{2^{m}}&\partial_y\phi_{2,1}I_{2^m}+\psi_{1}^{(m)}
			\end{matrix}\right).
		\end{eqnarray*}
		Since $\phi_{1,1}$ can be written as
		$\phi_{1,1}(x,y)=X_2^tA_1+X^tB_1+c_1$, using (\ref{pd}), the action of $\partial_x$ on $\phi_{1,1}$ is $\partial_x\,\phi_{1,1}(x,y)=X^t\left(N_{21}A_1\right)+N_{11}B_1$. Observing that $X^t\left(N_{21}A_1\right)I_{2^m}=\left(I_{2^m}\otimes\,X^t\right)\left(I_{2^m}\otimes\,\left(N_{21}A_1\right)\right)$ and using the fact that (\ref{e15}) is satisfied at order $m$ we obtain
		\begin{equation*}
			\partial_x\phi_{1,1}(x,y)I_{2^m}+\psi_1^{(m)}(x,y)=\left(I_{2^m}\otimes\,X^t\right)\left[I_{2^m}\otimes\,\left(N_{21}A_1\right)+D_1^{(m)}\right]+\left(N_{11}B_1\right)I_{2^m}+E_1^{(m)}.
		\end{equation*}
		In a similar way, we write $\partial_y\phi_{1,1}(x,y)I_{2^m}$, $\partial_x\phi_{2,1}(x,y)I_{2^m}$ and $\partial_y\phi_{2,1}(x,y)I_{2^m}+\psi_1^{(m)}$ in terms of $I_{2^m}\otimes X^t$ and take into account into the expression of $\psi_1^{(m+1)}$ given previously to get 
		\begin{equation*}
			\psi_1^{(m+1)}(x,y)=\left(I_{2^{m+1}}\otimes\,X^t\right)D_1^{(m+1)}+E_1^{(m+1)},
		\end{equation*}
		where
		\begin{align*}
			D_1^{(m+1)}&=\left( \begin{matrix}
				I_{2^m}\otimes\,\left(N_{2,1}A_1\right)+D_1^{(m)}&I_{2^m}\otimes\,\left(N_{2,1}A_2\right)\\
				I_{2^m}\otimes\,\left(N_{2,2}A_1\right)&I_{2^m}\otimes\,\left(N_{2,2}A_2\right)+D_1^{(m)}
			\end{matrix}\right)=H_1^{(m)}+I_2\otimes\,D_1^{(m)},\\
			E_1^{(m+1)}&=\left( \begin{matrix}
				I_{2^m}\otimes\,\left(N_{1,1}B_1\right)+E_1^{(m)}&I_{2^m}\otimes\,\left(N_{1,1}B_2\right)\\
				I_{2^m}\otimes\,\left(N_{1,2}B_1\right)&I_{2^m}\otimes\,\left(N_{1,2}B_2\right)+{E}_1^{(m)}
			\end{matrix}\right)=K_1^{(m)}+I_2\otimes\,E_1^{(m)}.
		\end{align*}
		In a similar way, we obtain  $\psi_2^{(m+1)}(x,y)$. 
	\end{proof}
	\begin{proposition}\label{l3} 
		Let $\left\{\mathbb{Q}_{n,m}\right\}_n$ be a family of vector polynomials, orthogonal with respect to the weight function $\rho_m=\rho\Phi^{\otimes\,m}$, satisfying the Pearson equation $div\left(\rho_m\otimes\,\Phi\right)=\rho_m\left(\psi_1^{(m)},\psi_2^{(m)}\right)$ with  the Neumann boundary condition (\ref{Neumannc}) where $\Phi\in\mathcal{M}_2\left(\mathcal{P}\right)$ and $\psi_i^{(m)}(X)=X^tD_i^{(m)}+E_i^{(m)}$ , $i=1,\,2$,  are polynomials given by (\ref{te4aa})-(\ref{te4bb})  with $det(D_1,\,D_2)\neq 0$  then
		\begin{equation}\label{e14}
			{\rm degree}\left(\left(\left(\Phi\nabla\right)\cdot\nabla\right)\,\mathbb{Q}_{n,m}+\psi_1^{(m)}\partial_x\mathbb{Q}_{n,m}+\psi_2^{(m)}\partial_y\mathbb{Q}_{n,m}\right)=n,\;\;\forall n,\; n\geq 1.
		\end{equation}
	\end{proposition}
	\begin{proof}
		If $degree(\Phi)=2$, follow the first part of the proof of Proposition \ref{prop3} to obtain (\ref{e14}). 
		If  {$degree(\Phi)< 2$}, use the Proposition \ref{l2} to have $D_i^{(m)}=I_2\otimes\,D_i^{(m-1)}$, $i=1,2$ and iterate to obtain 
		$D_i^{(m)}=I_{2^m}\otimes\,D_i$, $i=1,2$. Next taking into account Remark \ref{l1} yields $det\left(D_1^{(m)},D_2^{(m)}\right)=\left(-1\right)^{\left[2^{m-1}\right]}\left[det\left(D_1,\; D_2\right)\right]^{2^m}\neq\,0$, for $det\left(D_1,\,D_2\right)\neq\,0$. Finally follow the method described in the second part of the proof of the Proposition \ref{prop3}, to get $Rank\left(T_n^{(m)}\right)=2^m(n+1)$. Thus (\ref{e14}) follows.  
	\end{proof}

	\begin{theorem}\label{th1}\hspace{1 cm}\\
		Let $\Omega$ be a simply connected open subset of $\mathbb{R}^2$, $\rho$ a weight function on $\Omega$ and $\varPhi\in\mathcal{M}_{2}\left(\mathcal{P}_2\right)$ a symmetric and {positive definite} matrix 
		$\varPhi=\left(\begin{matrix}
			\phi_{1,1}&\phi_{1,2}\\
			\phi_{2,1}&\phi_{2,2}
		\end{matrix}\right)$ such that the Neumann type boundary  condition   
		\begin{equation}\label{bc}
			\lim\limits_{j\rightarrow\infty}1_{\partial\Omega_j}\left(\rho \Phi\,\nabla\,u\right)\cdot\overrightarrow{n_j}=0,\; \Omega_j=\Omega\cap B(O,j),\; {\rm for~all}\; u\in\mathcal{M}_{1,n}\left(\mathcal{P}\right)
		\end{equation} is satisfied  and 
		\begin{subequations}
			\begin{empheq}[left=\empheqlbrace]{align}
				\phi_{1,1}\partial_x\varPhi+\phi_{2,1}\partial_y\varPhi=\varPhi\nabla\left(\phi_{1,1},\phi_{2,1}\right),\label{te1a}\\
				\phi_{1,2}\partial_x\varPhi+\phi_{2,2}\partial_y\varPhi=\varPhi\nabla\left(\phi_{1,2},\phi_{2,2}\right).\label{te1b}
			\end{empheq}
		\end{subequations}

		Let $\left\{\mathbb{P}_n\right\}_n$, $\mathbb{P}_n\in \mathcal{M}_{n+1,1}\left(\mathcal{P}\right)$ be the family of vector polynomials, orthogonal with respect to the weight $\rho$  and $\left\{\mathbb{Q}_{n,m}\right\}_n$, $\mathbb{Q}_{n,m}\in\mathcal{M}_{2^{m},n+1}\left(\mathcal{P}\right)$, the family of polynomials of degree $n$ 
		\[\mathbb{Q}_{n,m}=\nabla\left(\mathbb{Q}_{n+1,m-1}\right),\;\;\;\mathbb{Q}_{n,0}=\mathbb{P}_n^t.\] The following properties are equivalent:
		\begin{enumerate}
			\item[(a)] There exist two polynomials $\psi_i(x,y)=X^tD_i+E_i,\,i=1,2$, such that $det\left(D_1,\,D_2\right)\neq\,0$ and
			\begin{equation}\label{te2a}
				div(\rho\varPhi)=\rho\left(\psi_{1},\,\psi_{2}\right).
			\end{equation}
			
			\item[(b)] There exist two polynomials  $\psi_i(x,y)=X^tD_i+E_i,\,i=1,2$, with  $det\left(D_1,\,D_2\right)\neq\,0$ such that for all positive integers $m$, $\left\{\mathbb{Q}_{n,m}\right\}_n$ is orthogonal with respect to $\rho_m=\rho\Phi^{\otimes\,m}$ and  
			\begin{equation}\label{te3}
				div(\rho_m\otimes \varPhi)=\rho_m\left(\psi_1^{(m)},\psi^{(m)}_2\right),
			\end{equation}
			where $\psi_i^{(m)},\,i=1,2$, are $2^m$-matrix polynomials defined by
			
			\begin{subequations}
				\begin{empheq}[left=\empheqlbrace]{align}
					&\psi^{(m)}_1=I_2\otimes\psi^{(m-1)}_1+\nabla\left(\phi_{1,1},\phi_{2,1}\right)\otimes I_{2^{m-1}},\label{te4a}\\
					&\psi^{(m)}_2=I_2\otimes\psi^{(m-1)}_2+\nabla\left(\phi_{1,2},\phi_{2,2}\right)\otimes I_{2^{m-1}}\label{te4b}
				\end{empheq}
			\end{subequations}
			
			with $\psi_1^{(0)}=\psi_1$ and $\psi_2^{(0)}=\psi_2$.
			\item[(c)] For fixed integers $n\geq 0$ and $m\geq 0$, $\mathbb{Q}_{n,m}$ satisfies the second order partial differential equation
			\begin{equation}\label{te2b}
				\left(\left(\varPhi\nabla\right)\cdot\nabla\right)\mathbb{Q}_{n,m}+\psi_{1}^{(m)}\partial_x\mathbb{Q}_{n,m}+\psi_{2}^{(m)}\partial_y\mathbb{Q}_{n,m}+\mathbb{Q}_{n,m}\Lambda_{n+m,m}=0,
			\end{equation}
			where the matrix polynomials $\psi_{1}^{(m)}$, $\psi_{2}^{(m)}$ are defined in (\ref{te4a}) and (\ref{te4b}),  $\Lambda_{n+m,m}$ is the $n+m+1$-matrix defined by 
			{\small	\begin{align*}
					\lefteqn{G_{n,m,n}\Lambda_{n+m,m}}\\
					&=-\left[\left(L_{n-1}^{(m)*}\right)^t\left(\left(A_1,\,2A_2,\,A_3\right)\otimes\,I_{2^m(n-1)}\right)N_n^{(m)*}+\left(I_{2^m}\otimes L_{n-1}^t\right)\left(\left(D_1^{(m)},\,D_2^{(m)}\right)\otimes\,I_n\right)N_n^{(m)}\right]\\
					&\times G_{n,m,n},
			\end{align*}}
			where	$G_{n,m,n}$ is the leading coefficient of $\mathbb{Q}_{n,m}$  given in (\ref{e8a}) and $\Lambda_{n,0}\equiv\Lambda_n$.
			\item[(d)] For an integer $n\geq 1$ the following Rodrigues formula holds 
			\begin{eqnarray}\label{Rodrigues}
				\mathbb{P}^t_n=\frac{(-1)^n}{\rho}div^{(n)}\left[\rho\varPhi^{\otimes\,n}\right]R_n,
			\end{eqnarray}
			where
			$R_n=\left(\nabla^{(n)}\mathbb{P}_n^t\right)\prod_{j=0}^{n-1}\left(\Lambda_{n,j}\right)^{-1}$ and   $\Lambda_{n,j},\,j=0..n-1$ are the matrices in (\ref{te2b}).
			\item[(e)] There exist $A^{n,m}_{n+1}$, $A^{n,m}_{n}$ and $A^{n,m}_{n-1}$, where $A^{n,m}_{p}=\left(\begin{matrix}
				A^{n,m}_{p,1}&\\
				A^{n,m}_{p,2}	
			\end{matrix}\right)$, with $A^{n,m}_{p,i}\in \mathcal{M}_{p+1,\,n+1}\left(\mathbb{R}\right)$, $i=1$ ,$2$ and $p=n-1\,, n\, ,n+1$ such that
			\begin{equation}\label{SR}
				\left(\varPhi\otimes\,I_{2^m}\right) \mathbb{Q}_{n-1,m+1}=(I_2\otimes\mathbb{Q}_{n+1,m})A^{n,m}_{n+1}
				+(I_2\otimes\mathbb{Q}_{n,m})A^{n,m}_{n}+(I_2\otimes\mathbb{Q}_{n-1,m})A_{n-1}^{n,m}
			\end{equation}
			and $A^{n,m}_{n-1}$ invertible.
		\end{enumerate}
	\end{theorem}
	
	\begin{proof}
		We organize the proof in the following scheme:\newline
		\begin{enumerate}
			\item[Step 1] $(a)\Rightarrow\,(b)\Rightarrow\,(c)\Rightarrow (a)$ which is equivalent to $(a)\Leftrightarrow\,(b)\Leftrightarrow\,(c)$.
			\item [Step 2] $(b)$ and $(c) \Rightarrow\,(d)\Rightarrow\,(a)$ which, taking into account Step 1, is equivalent to $(d)\Leftrightarrow (a)$.
			\item [Step 3] $(b)\Rightarrow\,(e)\Rightarrow\,(a)$ which  taking into account Step 1 is equivalent to $(e)\Leftrightarrow\,(a)$.
		\end{enumerate}
		\begin{enumerate}
			\item[Step 1] $(a)\Rightarrow\,(b)\Rightarrow\,(c)\Rightarrow (a)$, which is equivalent to $(a)\Leftrightarrow\,(b)\Leftrightarrow\,(c)$.
			
			\vspace{0.2cm} [Step 1.1] $(a)\Rightarrow\,(b)$ \\
			We assume that property (a) is satisfied, and we show by induction on $m$ that (b) is
			satisfied for $m\geq 0$. \\
			Let us show by induction on $m$ that (\ref{te3}) is
			satisfied for $m\geq 0$.   Assume that (\ref{te3}) is satisfied for a fixed integer $m\geq\,1$. Let us first observe from the associativity of $\otimes$ that 
			\begin{equation}\label{te5}
				\rho_{m}\otimes \varPhi=\rho\varPhi\otimes\varPhi^{\otimes\,m}=\left(\begin{matrix}
					\rho\phi_{1,1}\varPhi^{\otimes\,m}&\rho\phi_{1,2}\varPhi^{\otimes\,m}\\
					\rho\phi_{2,1}\varPhi^{\otimes\,m}&\rho\phi_{2,2}\varPhi^{\otimes\,m}
				\end{matrix}\right).  
			\end{equation} 
			Therefore, the induction hypothesis $div(\rho_m\otimes \varPhi)=\rho_m(\psi_1^{(m)},\,\psi_2^{(m)})$ is equivalent to 
			\begin{subequations}
				\begin{empheq}[left=\empheqlbrace]{align}
					&\partial_x\left(\rho\phi_{1,1}\varPhi^{\otimes\,m}\right)+\partial_y\left(\rho\phi_{2,1}\varPhi^{\otimes\,m}\right)=\rho_m\psi_1^{(m)},\label{te7a}\\
					&\partial_x\left(\rho\phi_{1,2}\varPhi^{\otimes\,m}\right)+\partial_y\left(\rho\phi_{2,2}\varPhi^{\otimes\,m}\right)=\rho_m\psi_2^{(m)}\label{te7b}.
				\end{empheq}
			\end{subequations}
			Replacing $m$ by $m+1$ in (\ref{te5}) and applying the divergence operator on both sides we obtain
			\begin{equation*}
				div(\rho_{m+1}\otimes \Phi)=\left(\partial_x\left(\rho\phi_{1,1}\Phi^{\otimes\,m+1}\right)+\partial_y\left(\rho\phi_{2,1}\Phi^{\otimes\,m+1}\right),\partial_x\left(\rho\phi_{1,2}\Phi^{\otimes\,m+1}\right)+\partial_y\left(\rho\phi_{2,2}\varPhi^{\otimes\,m+1}\right)\right).
			\end{equation*}
			Observing that for a polynomial $p\in\mathcal{P}$, $p\Phi^{\otimes\,m+1}=\Phi\otimes p\varPhi^{\otimes\,m}$, taking  $p=\phi_{1,1}$ and using the product rule (\ref{p1e1a}) with $A=\varPhi$ and $B=\rho\phi_{1,1}\Phi^{\otimes\,m}$, we have $\partial_x\left(\rho\phi_{1,1}\varPhi^{\otimes\,m+1}\right)=\varPhi\otimes\partial_x\left(\rho\phi_{1,1}\Phi^{\otimes\,m}\right)+\rho\phi_{1,1}\partial_x\varPhi\otimes \Phi^{\otimes\,m}.$ In a similar way, we have
			$\partial_y\left(\rho\phi_{2,1}\Phi^{\otimes\,m+1}\right)=\Phi\otimes{\partial_y}\left(\rho\phi_{2,1}\Phi^{\otimes\,m}\right)+\rho\phi_{2,1}{\partial_y}\varPhi\otimes \Phi^{\otimes m}.$ Adding both and using  (\ref{te1a}) and  ( \ref{te7a}), we obtain 
			\begin{eqnarray*}
				\partial_x\left(\rho\phi_{1,1}\varPhi^{\otimes\,m+1}\right)+\partial_y\left(\rho\phi_{2,1}\Phi^{\otimes\,m+1}\right)
				=\Phi\otimes\left(\rho\Phi^{\otimes m}\psi_1^{(m)}\right)+\left(\Phi\nabla(\phi_{1,1},\phi_{2,1})\right)\otimes \rho\Phi^{\otimes m}
			\end{eqnarray*}
			which can be rewritten  as   
			\begin{eqnarray*}
				\partial_x\left(\rho\phi_{1,1}\Phi^{\otimes\,m+1}\right)+\partial_y\left(\rho{\phi_{2,1}}\varPhi^{\otimes\,m+1}\right)
				=\left(\Phi\,I_2\right)\otimes\left(\rho\Phi^{\otimes m}\psi_1^{(m)}\right)+\left(\Phi\nabla(\phi_{1,1},\phi_{2,1})\right)\otimes\left( \rho\Phi^{\otimes m}I_{2^m}\right)
			\end{eqnarray*}
			since $\Phi$ and $\Phi^{\otimes m}$ are  $2$ and $2^m$ matrices. Taking consecutively  $\left(A,B,C,D\right)=\left(\Phi,\,\rho\Phi^{\otimes m},\,I_2,\psi_1^{(m)}\right)$  and $\left(A,B,C,D\right)=\left(\Phi,\,\rho\Phi^{\otimes m},\,\nabla\left(\phi_{1,1},\,\phi_{2,1}\right),\,I_{2^m}\right)$ in  (\ref{p1e2}), we have 
			\begin{eqnarray*}
				\left(\Phi\,I_2\right)\otimes\left(\rho\Phi^{\otimes m}\psi_1^{(m)}\right)&=&
				\left(\Phi\otimes\rho\Phi^{\otimes m}\right)\left(I_2\otimes\psi_1^{(m)}\right)\\ \left(\Phi\nabla(\phi_{1,1},\phi_{2,1})\right)\otimes\left( \rho\Phi^{\otimes m}I_{2^m}\right)&=&\left(\varPhi\otimes \rho\Phi^{\otimes m}\right)\left(\nabla(\phi_{1,1},\phi_{2,1})\otimes I_{2^m}\right).
			\end{eqnarray*}
			Therefore 
			\begin{eqnarray*}
				\partial_x\left(\rho\phi_{1,1}\varPhi^{\otimes\,m+1}\right)+\partial_y\left(\rho\phi_{2,1}\varPhi^{\otimes\,m+1}\right)
				&=&\rho_{m+1}\left(I_2\otimes\psi_1^{(m)}+\nabla(\phi_{1,1},\phi_{2,1})\otimes I_{2^m}\right),\\
				&=&\rho_{m+1}\psi_1^{(m+1)}.
			\end{eqnarray*}
			In a similar way, one proves that 
			\begin{eqnarray*}
				\partial_x\left(\rho\phi_{1,2}\Phi^{\otimes\,m+1}\right)+\partial_y\left(\rho\phi_{2,2}\Phi^{\otimes\,m+1}\right)=\rho_{m+1}\psi_2^{(m+1)}.
			\end{eqnarray*}
			So, $div(\rho_{m+1}\otimes\varPhi)=\rho_{m+1}\left(\psi_1^{(m+1)},\,\psi_2^{(m+1)}\right).$\\
			Let us prove that for a fixed $m\geq 1$ the family of polynomials 
			$\left\{Q_{n,m}\right\}_n$ is orthogonal with respect to $\rho_m=\rho \Phi^{\otimes m}$.
			Taking $M=\mathbb{Q}_{n+1,m-1}^t$, $N=\mathbb{Q}_{j+1,m-1}$ in (\ref{p2e1}), replacing $A^{\otimes\,m }$ by $\rho\Phi^{\otimes\,m }$ and taking into account the definition of $\rho_{m}$, we have
			\begin{eqnarray*}
				\lefteqn{div\left[\left(I_2\otimes\,\mathbb{Q}_{n+1,m-1}^t\right)\rho\Phi^{\otimes\,m}\nabla\,\mathbb{Q}_{j+1,m-1}\right]}&\\
				&=\mathbb{Q}_{n+1,m-1}^tdiv(\rho_{m-1}\otimes\Phi\nabla \mathbb{Q}_{j+1,m-1})+\left(\partial_x\mathbb{Q}_{n+1,m-1}^t,{\partial_y}\mathbb{Q}_{n+1,m-1}^t\right)\rho_m\nabla \mathbb{Q}_{j+1,m-1}
			\end{eqnarray*}
			
			Integrating both sides on $\Omega$ we obtain
			\begin{eqnarray*}
				\int_{\Omega}\left(\nabla\mathbb{Q}_{n+1,m-1}^t\right)^t\rho_m\nabla\mathbb{Q}_{j+1,m-1}dxdy
				&=&\int_{\Omega}div\left[\left(I_2\otimes\,\mathbb{Q}_{n+1,m-1}^t\right)\rho\Phi^{\otimes\,m}\nabla\,\mathbb{Q}_{j,m-1}\right]dxdy\\
				&&-\int_{\Omega}\mathbb{Q}_{n+1,m-1}div(\left(\rho_{m-1}\otimes\Phi\right)\nabla \mathbb{Q}_{j+1,m-1})dxdy.
			\end{eqnarray*}
			Taking into account the boundary condition (\ref{bc}) as well as the second item of Proposition \ref{p2a}, we obtain
			\begin{eqnarray*}
				\int_{\Omega}\left(\nabla\mathbb{Q}_{n+1,m-1}^t\right)^t\rho_m\mathbb{Q}_{j+1,m-1}dxdy=-\int_{\Omega}\mathbb{Q}_{n+1,m-1}^tdiv(\left(\rho_{m-1}\otimes\Phi\right)\nabla \mathbb{Q}_{j,m-1})dxdy.
			\end{eqnarray*}
			Using (\ref{p2e2}) with $B=\Phi$, $N=\mathbb{Q}_{j+1,m-1}$, replacing $m$ by $m-1$ and $\,A^{\otimes\,m-1}$ by $\rho\Phi^{\otimes\,m-1}$ and taking into account the fact that $div\left(\rho_{m-1}\otimes \Phi\right)=\rho_{m-1}\left(\psi_1^{(m-1)},\,{\psi_2^{(m-1)}}\right)$ we obtain
			\begin{align*}
				\lefteqn{\int_{\Omega}\left(\nabla\mathbb{Q}_{n+1,m-1}^t\right)^t\rho_m\nabla\mathbb{Q}_{j+1,m-1}dxdy}&\\
				&=\int_{\Omega}\mathbb{Q}_{n+1,m-1}^t	\rho_{m-1}\left[\left[\left(\left(\Phi\nabla\right)\cdot\nabla\right)\right]\mathbb{Q}_{j+1,m-1}+\left(\psi_1^{(m-1)},\,\psi_2^{(m-1)}\right)\nabla\mathbb{Q}_{j+1,m-1}\,\right]dxdy\\
				&=H_{n+1}\delta_{n,j}.
			\end{align*}
			
			Here $H_{n+1}$ is an $(n+1, n+1)$-invertible matrix because the family $\left\{\mathbb{Q}_{n+1,m-1}\right\}$ is orthogonal with respect to $\rho_{m-1}$ and  the matrix polynomial $\left[\left[\left(\left(\Phi\nabla\right)\cdot\nabla\right)\right]\mathbb{Q}_{j+1,m-1}+\left(\psi_1^{(m-1)},\,\psi_2^{(m-1)}\right)\nabla\mathbb{Q}_{j+1,m-1}\,\right]$ is of total degree $j+1$, thanks to Proposition \ref{l3}.
			
			\vspace{0.1cm}[Step 1.2] $(b)\Rightarrow\,(c)$   \\
			We assume $(b)$ and fix two non-negative integers $n$ and $m$. Multiplying the expansion 
			
			\begin{equation*}
				\left(\left(\Phi\nabla\right)\cdot\nabla\right)\mathbb{Q}_{n,m}+\left(\psi_{1}^{(m)},\,\psi_{2}^{(m)}\right)\nabla \mathbb{Q}_{n,m}=\sum_{j=0}^{n}\mathbb{Q}_{j,m}A_{j+m,m}
			\end{equation*}
			from the left side by $\rho_{m}$,  and using (\ref{te3}) and then (\ref{p2e2}) with $A^{\otimes\,m}=\rho \Phi^{\otimes\,m}$, $B=\Phi$ and $N=\mathbb{Q}_{n,m}$, we obtain
			\begin{equation*}
				div\left[\left(\rho_{m}\otimes\Phi\right)\nabla\mathbb{Q}_{n,m}\right]=\sum_{j=0}^{n}\rho_{m}\mathbb{Q}_{j,m}A_{j+m,m}.
			\end{equation*}
			Multiplying from the left side by ${\mathbb{Q}_{k,m}^t}$, $0\leq\,k\leq\,n$ and taking into account (\ref{p2e1}) with $M={\mathbb{Q}_{k,m}^t}$, $A^{\otimes\,m+1}=\rho\Phi^{\otimes\,m+1}=\rho_{m+1}$ and $N=\mathbb{Q}_{n,m}$, we obtain
			\begin{equation*} div\left[\left(I_2\otimes\mathbb{Q}_{k,m}^t\right)\rho_{m+1}\nabla\mathbb{Q}_{n,m}\right]-\left(\nabla\mathbb{Q}_{k,m}\right)^t\rho_{m+1}\nabla\mathbb{Q}_{n,m}=\sum_{j=0}^{n}\mathbb{Q}_{k,m}^t\rho_{m}\mathbb{Q}_{j,m}A_{j+m,m}.
			\end{equation*}
			
			Integrating both sides on $\Omega$ and taking into account the boundary condition (\ref{bc}) as well as the Proposition \ref{p2a}, we obtain 
			
			\begin{equation*}
				\int_{\Omega}\mathbb{Q}_{k,m}^t\rho_{m}\mathbb{Q}_{k,m}dxdyA_{k+m,m}=-\int_{\Omega}\left(\nabla\mathbb{Q}_{k,m}\right)^t\rho_{m+1}\nabla\mathbb{Q}_{n,m}dxdy.
			\end{equation*}
			Since the family $\left\{\nabla\mathbb{Q}_{n,m}\right\}_{n\geq\,1}$ is orthogonal with respect to $\rho_{m+1}$, $A_{k+m,m}=0$, $k<n$. Therefore,
			
			\begin{equation*}
				\left[\left(\Phi\nabla\right)\cdot\nabla\right]\mathbb{Q}_{n,m}+\left(\psi_{1}^{(m)},\,\psi_{2}^{(m)}\right)\nabla\mathbb{Q}_{n,m}=\mathbb{Q}_{n,m}A_{n+m,m}.
			\end{equation*}
			Write $\mathbb{Q}_{n,m}=\left(I_{2^m}\otimes\,X_n^t\right)G_{n,m,n}$+lower terms and identify both sides using the relation (\ref{e9}) to obtain $G_{n,m,n}\Lambda_{n+m,m}=-T_n^{(m)}G_{n,m,n}$.
			
			\vspace{0.1cm}[Step~1.3] $(c)\Rightarrow\,(a)$ \\
			Taking $m=0$ into  (\ref{te2b}) and multiplying by $\rho$ we have
			\begin{equation*}
				\left[\left(\rho\Phi\nabla\right)\cdot\nabla\right]\mathbb{P}_n^t+\rho\left(\psi_{1},\psi_{2}\right)\nabla\mathbb{P}_n^t+\rho\mathbb{P}_n^t\Lambda_{n,0}=0.
			\end{equation*}
			Taking into account the relation (\ref{p2e2}) with $m=0$, $B=\rho\Phi$ and $N=\mathbb{P}_n^t$ we obtain
			\begin{equation*}
				div\left(\rho\Phi\nabla\mathbb{P}_n^t\right)+\left[\rho\left(\psi_{1},\,\psi_{2}\right)-div\left(\rho\Phi\right)\right] \nabla\mathbb{P}_n^t+\rho\mathbb{P}_n^t\Lambda_{n,0}=0.
			\end{equation*}
			Integrating both sides on the domain $\Omega$ and using the boundary condition (\ref{bc}), Proposition \ref{p2a} as well as the orthogonality of the family  $\left\{\mathbb{P}_n^t\right\}_n$ with respect to $\rho$ ,we obtain
			\begin{equation*}
				\int_{\Omega}\left[\rho\left(\psi_{1},\,\psi_{2}\right)-div\left(\rho\Phi\right)\right] \nabla\mathbb{P}_n^tdxdy=0.
			\end{equation*}
			
			Since $\{\nabla\mathbb{P}_n^t\}_{n\geq\,1}$ is a basis of $\mathcal{P}\times\,\mathcal{P}$, we get
			
			\begin{equation*}
				\int_{\Omega}\left[\rho\left(\psi_{1},\,\psi_{2}\right)-div\left(\rho\Phi\right)\right] \mathbb{P}dxdy=0, \; \text{for all } \mathbb{P}=(p,q)^t\in\mathcal{P}\times\mathcal{P}.
			\end{equation*}
			Taking $\mathbb{P}=(p,0)^t$ and using the fact that $div(\rho\Phi)=\left(\partial_x\left(\rho\phi_{1,1}\right)+\partial_y\left(\rho\phi_{2,1}\right),\partial_x\left(\rho\phi_{1,2}\right)+\partial_y\left(\rho\phi_{2,2}\right)\right),$ we obtain
			\begin{equation*}
				\int_{\Omega}{\left(\rho\psi_{1}-\partial_x\left(\rho\phi_{1,1}\right)-\partial_y\left(\rho\phi_{2,1}\right)\right)\rho^{-1}}p\rho dxdy=0, p\in\mathcal{P}. 
			\end{equation*}
			{That is
				\begin{equation}\label{tpeq}
					\int_{\Omega}{\left(\rho\psi_{1}-\partial_x\left(\rho\phi_{1,1}\right)-\partial_y\left(\rho\phi_{2,1}\right)\right)\rho^{-1}}u\rho dxdy=0, \;u\in L^2\left(\Omega,\rho\right),
				\end{equation}by Theorem \ref{NN}. Therefore, from the Hölder inequality, we obtain 
				\begin{equation*}
					\left\lvert\int_{\Omega}\tfrac{\partial_x(\rho\phi_{1,1})+\partial_y(\rho\phi_{2,1})}{\rho}u\rho\,dx\right\rvert=\left\lvert\int_{\Omega}\psi_1u\rho\,dx\right\rvert\leq\,C\lVert u\rVert_{L^2\left(\Omega,\,\rho\right)},u\in L^2\left(\Omega,\,\rho\right),
				\end{equation*}where $C=\lVert\psi_1\rVert_{L^2\left(\Omega,\,\rho\right)}$. So, the linear functional $L(u)=\int_{\Omega}\tfrac{\partial_x(\rho\phi_{1,1})+\partial_y(\rho\phi_{2,1})}{\rho}u\rho\,dx,\,u\in L^2\left(\Omega,\,\rho\right)$ is continuous. Hence, by the Riesz representation theorem, $\tfrac{\partial_x(\rho\phi_{1,1})+\partial_y(\rho\phi_{2,1})}{\rho}\in L^2\left(\Omega,\,\rho\right)$.}
			Therefore $$\left(\partial_x\left(\rho\phi_{1,1}\right)+\partial_y\left(\rho\phi_{2,1}\right)-\rho\psi_{1}\right)\rho^{-1}\in L^2(\Omega,\rho).$$ {Combining with (\ref{tpeq}), we obtain} $\partial_x\left(\rho\phi_{1,1}\right)+\partial_y\left(\rho\phi_{2,1}\right)=\rho\psi_{1}$. In a similar way, taking $\mathbb{P}=(0,q),\,q\in\mathcal{P}$, we obtain $\partial_x\left(\rho\phi_{1,2}\right)+\partial_y\left(\rho\phi_{2,2}\right)=\rho\psi_{2}$. Thus $div\left(\rho{\Phi}\right)=\rho\left(\psi_{1},\,\psi_{2}\right)$.
			\item [Step 2] $(b) \text{ and } (c)\Rightarrow\,(d)\Rightarrow\,(a)$ which, taking into account Step 1, is equivalent to $(d)\Leftrightarrow (a)$.\\\,   [Step 2.1]	$(b)$ and $(c) \Rightarrow\,(d)$
			From (\ref{te2b}),
			\begin{equation*}
				\left[\left(\Phi\nabla\right)\cdot\nabla\right]\nabla^{(m)}\mathbb{P}_{n+m}^t+\left(\psi_{1}^{(m)},\psi_{2}^{(m)}\right)\nabla\nabla^{(m)}\mathbb{P}_{n+m}^t+\nabla^{(m)}\mathbb{P}_{n+m}^t\Lambda_{n+m,m}=0.
			\end{equation*}
			Replace $n$ by $n-m$, multiply the obtained equation by $\rho_{m}$ from the left side and use (\ref{p2e2}), with $A^{\otimes m}=\rho_{m}$, $B=\Phi$ and $N=\mathbb{P}_{n}^t$, to obtain
			
			\begin{equation*}
				div\left(\left(\rho_{m}\otimes\Phi\right)\nabla^{(m+1)}\mathbb{P}_n^t\right)+\rho_{m}\nabla^{(m)}\mathbb{P}_n^t\Lambda_{n,m}=0.
			\end{equation*}
			Replace $m$ by $m-1$ and  iterate the resulting relation to obtain
			\begin{equation*}
				div^{(m)}\left(\rho_{m}\nabla^{(m)}\mathbb{P}_n^t\right)=(-1)^n\rho\mathbb{P}_n^t\Lambda_{n,m-1}\Lambda_{n,m-2}\dots\,\Lambda_{n,0}.
			\end{equation*}
			Take $m=n$ to get the result.\\
			\,[Step~2.2] $(d)\Rightarrow\,(a)$ Take  $n=1$ in  (d) to obtain
			\begin{equation*}
				div(\rho\Phi)\nabla\,\mathbb{P}_1^t=-\rho\mathbb{P}_1^t\Lambda_{1,0}.
			\end{equation*}
			Since $\mathbb{P}_1$ is monic, ${\mathbb{P}}_1^t(x,y)=X^t+G_{1,0,0}=(x+g_{10},\,y+g_{01})$ and $\nabla\,{\mathbb{P}}_1^t=I_2$. From (\ref{te2b}), $$G_{1,0,1}\Lambda_{1,0}=-\left(L_0^{(0)}\right)^t\left(D_1,\,D_2\right)N_1^{(0)}G_{1,0,1}.$$
			That is $\Lambda_{1,0}=-(D_1,D_2)$ for $G_{1,0,1}=L_0^{(0)}=N_1^{(0)}=I_2$ see Proposition \ref{prop4} and Eq. (\ref{e8a}). Therefore writing $D_i=(d_{i1},\,d_{i2}),\,i=1,2$, we obtain
			$div(\rho\Phi)=\rho\mathbb{P}_1^t(D_1,D_2)=\rho\left(\psi_{1},\psi_{2}\right)$
			with $\psi_1(x,y)=xd_{11}+yd_{21}+g_{10}d_{11}+g_{01}d_{21}$ and $\psi_2(x,y)=xd_{12}+yd_{22}+g_{10}d_{12}+g_{01}d_{22}$ which are polynomials of total degree one, for $det(D_1,D_2)\neq\,0$.
			
			\item [Step 3] $(b)\Rightarrow\,(e)\Rightarrow\,(a)$ which  taking into account Step 1 is equivalent to $(e)\Leftrightarrow\,(a)$.\\
			\,[Step~3.1] $(b)\Rightarrow\,(e)$
			Use the relation (\ref{cprod}) to rewrite the right side of the expansion
			\begin{equation*}
				\left(\Phi\otimes\,I_{2^m}\right)\nabla\,\mathbb{Q}_{n,m}=\sum_{k=0}^{n+1}\left(I_{2}\otimes\mathbb{Q}_{k,m}\right)A_k^{n,m}
			\end{equation*}
			and multiply, for a fixed $0\leq\,j\leq n$, both sides, from the left, by 
			\begin{equation*}
				\left(\begin{matrix}
					\mathbb{Q}_{j,m}^t&0\\	0&\mathbb{Q}_{j,m}^t
				\end{matrix}\right)\left(\begin{matrix}
					\rho_{m}&0\\ 0&\rho_{m}
				\end{matrix}\right).
			\end{equation*}
			Observe that $\rho_{m+1}=\rho_{m}\left(\Phi\otimes\,I_{2^m}\right)$ to have 
			\begin{equation*}
				\left(\begin{matrix}
					\mathbb{Q}_{j,m}^t&0\\	0&\mathbb{Q}_{j,m}^t
				\end{matrix}\right)\rho_{m+1}\mathbb{Q}_{n-1,m+1}=\sum_{k=0}^{n+1}\left(\begin{matrix}\mathbb{Q}_{j,m}^t\rho_{m}\mathbb{Q}_{k,m}&0\\
					0&\mathbb{Q}_{j,m}^t\rho_{m}\mathbb{Q}_{k,m}\end{matrix}\right)A_k^{n,m}.
			\end{equation*}
			Integrate both sides on the domain $\Omega$ and use the fact $\left\{\mathbb{Q}_{n,m}\right\}_{n}$ (resp. $\left\{\mathbb{Q}_{n-1,m+1}\right\}_{n}$) is orthogonal with respect to $\rho_{m}$ (resp. $\rho_{m+1}$) to obtain
			\begin{eqnarray*}
				\left(\begin{matrix}H_{j,m}&0\\
					0&H_{j,m}\end{matrix}\right)A_j^{n,m}&=&\int_{\Omega}	\left(\begin{matrix}
					\mathbb{Q}_{j,m}^t&0\\	0&\mathbb{Q}_{j,m}^t
				\end{matrix}\right)\rho_{m+1}\mathbb{Q}_{n-1,m+1}dxdy\\
				&=&\begin{cases}
					&0\;{\rm if}\; 0\leq\,j<n-1\\
					&M_{n-1,m}\;{\rm if}\; j=n-1,
				\end{cases} 
			\end{eqnarray*}
			where $H_{j,m}=\int_{\Omega}\mathbb{Q}^t_{j,m}\rho_{m}\mathbb{Q}_{j,m}dxdy$ and $M_{n-1,m}$ are invertible matrices. Therefore, $A_{j}^{n,m}=0$, for $j=0,..,n-2$ and $A_{n-1}^{n,m}$ is an invertible matrix.\\ 
			\,[Step~3.2] $(e)\Rightarrow\,(a)$
			Let
			\[\frac{div(\rho\Phi)}{\rho}=\sum_{n=0}^{\infty}\mathbb{P}^t_nA_n\] be the formal Fourier expansion of the function $\frac{div(\rho\Phi)}{\rho}$ in the system $\{\mathbb{P}_n\}_{n\in \mathbb{N}}$. Since the system is orthogonal with respect to $\rho$, the coefficients $A_n$ are given by
			\[\int_{\Omega}\mathbb{P}_n\mathbb{P}_n^tdxdy\,A_n=\int_{\Omega}\mathbb{P}_n div(\rho\Phi)dxdy.\]
			Take into account the relation
			\begin{equation*}
				div\left(\rho\Phi\otimes\mathbb{P}_n\right)=\mathbb{P}_ndiv\left(\rho\Phi\right)+\left(\partial_x\mathbb{P}_n,\partial_y\mathbb{P}_n\right)\rho\Phi,
			\end{equation*} 
			obtained by direct computation,  as well as the boundary condition (\ref{bc}) to obtain
			\begin{equation*}
				\int_{\Omega}\mathbb{P}_n\mathbb{P}_n^t\rho\,dxdy\,A_n=-\int_{\Omega}\left(\partial_x\mathbb{P}_n,\partial_y\mathbb{P}_n\right)\rho\Phi\,dxdy.
			\end{equation*}
			Transpose and take into account the fact that the matrix $\Phi$ is symmetric to obtain
			
			\begin{equation*}
				A_n^t\int_{\Omega}\rho\mathbb{P}_n\mathbb{P}_n^tdxdy=-\int_{\Omega}\rho\Phi\nabla\mathbb{P}_n^t\,dxdy.
			\end{equation*}
			Therefore $A_0=0$. Use the assumption to get $A_n=0$, $n>1$ and
			
			\begin{align*}
				A_1^t\int_{\Omega}\rho\mathbb{P}_1\mathbb{P}_1^tdxdy=&-\int_{\Omega}\left(\begin{matrix}
					\rho\mathbb{P}_2^t&0\\
					0&\rho\mathbb{P}_2^t
				\end{matrix}\right)A_2^{1,0}dxdy-\int_{\Omega}\left(\begin{matrix}
					\rho\mathbb{P}_1^t&0\\
					0&\rho\mathbb{P}_1^t
				\end{matrix}\right)A_1^{1,0}dxdy\\ &-\int_{\Omega}\left(\begin{matrix}
					\rho\mathbb{P}_0^t&0\\
					0&\rho\mathbb{P}_0^t
				\end{matrix}\right)A_0^{1,0}dxdy.
			\end{align*}
			Use the fact that the system $\{\mathbb{P}_n\}_n$ is orthogonal with respect to $\rho$ to obtain 
			\begin{equation*}
				A_1^t\int_{\Omega}\rho\mathbb{P}_1\mathbb{P}_1^tdxdy=-
				\left(\begin{matrix}
					a_0&0\\
					0&a_0
				\end{matrix}\right){A_0^{1,0}},\;\;a_0=\int_{\Omega}\rho\mathbb{P}_0^tdxdy.
			\end{equation*}
			Therefore $A_1=\left(\begin{matrix}
				b_0&0\\
				0&b_1
			\end{matrix}\right)$, with $b_0b_1\neq\,0$  for $\int_{\Omega}\rho\mathbb{P}_1\mathbb{P}_1^tdxdy$ and $A_0^{1,0}$ are invertible. Thus,
			\begin{eqnarray*}
				{div(\rho\Phi)}&=&-{\rho}\mathbb{P}_1^t\left(\begin{matrix}
					b_0&0\\
					0&b_1
				\end{matrix}\right)\\
				&=&\left(\psi_{1},\,\psi_{2}\right),
			\end{eqnarray*}
			where $\psi_{1}$ and $\psi_{2}$ are polynomials of degree 1 of the variable $x$ and $y$ respectively.
		\end{enumerate}
	\end{proof}
	
	We now propose the following definition of classical orthogonal polynomials in two variables. 
	\begin{definition}\label{cweight} Let $\Omega$ be a simply connected open subset of $\mathbb{R}^2$. A weight function $\rho$ on $\Omega$ is classical if there exists a symmetric {and positive definite} {$2$}-matrix  vector polynomial $\Phi$ of total degree at most $2$ satisfying the Neumann boundary condition (\ref{bc}) and the differential system (\ref{te4a})-(\ref{te4b}), and there exist two polynomials $\psi_i(x,y)=X^tD_i+E_i,\,i=1,2$ of total degree $1$ such that $det(D_1,D_2)\neq\,0$ and  $$div\left(\rho\Phi\right)=\rho\left(\psi_1,\,\psi_2\right).$$   
	\end{definition}
	\begin{definition}\label{cop} Let $\Omega$ be a simply connected open subset of $\mathbb{R}^2$ and $\rho$ a  weight function on $\Omega$. A family of vector polynomials in two variables  $\{\mathbb{P}_n\}$, orthogonal with respect to $\rho$ is classical if $\rho$ is classical. 
	\end{definition}
	
	\begin{remark}\label{remark} \hspace{1cm}\\
		\begin{enumerate}
			\item Since $\Phi$ is symmetric and positive definite, the successive tensor product  $\Phi^{\otimes m},m\geq 1$ involved in the characterization theorem is symmetric and positive definite. 
			\item We can observe that $det(\Phi^{\otimes m})=(det(\Phi))^{2^{m-1}m},m\geq 1.$
			\item If the family $\{\mathbb{P}_n\}_n$ is monic, that is $\mathbb{P}_n=X_n+\dots$, the Rodrigues formula (\ref{Rodrigues}) reads
			\[\mathbb{P}_n^t=\tfrac{(-1)^n}{\rho}div^{(n)}\left(\rho\Phi^{\otimes n}\right)\nabla^{(n)}X_n^t\prod_{j=0}^{n-1}\Lambda_{n,j}^{-1}.\]
		\end{enumerate}
	\end{remark}
	
	\begin{remark}\label{GR}\hspace{1cm}\\
		\begin{enumerate}
			\item If $m=1$, the orthogonality condition in Theorem\,\ref{th1}(b) reads as $\int_{\Omega}\left(\nabla\mathbb{P}_{n+1}^t\right)^t\rho\Phi\nabla\mathbb{P}_{j+1}^t=0,n\neq j$ which is the weighted version of \cite[Theorem 4.7]{AMFPP} with $h=1$.
			\item If $m=2$, the orthogonality condition in Theorem\,\ref{th1}(b) reads as $\int_{\Omega}\left(\nabla^{(2)}\mathbb{P}_{n+2}^t\right)^t\rho\Phi^{(2)}\nabla^{(2)}\mathbb{P}_{j+2}^t=0,n\neq j$, where $\Phi^{(2)}=\Phi\otimes\Phi$. Observing that \[\nabla^{\{2\}}\mathbb{P}_{n+2}^t=\left(\begin{matrix}
				\partial_x^2\mathbb{P}_{n+2}^t\\
				2\partial_{xy}^2\mathbb{P}_{n+2}^t\\
				\partial_y^2\mathbb{P}_{n+2}^t\end{matrix}\right)=\left(\begin{matrix}
				1\,0\,0\,0\\
				0\,1\,1\,0\\
				0\,0\,0\,1
			\end{matrix}\right)\left(\begin{matrix}
				\partial_x^2\mathbb{P}_{n+2}^t\\
				\partial_{xy}^2\mathbb{P}_{n+2}^t\\
				\partial_{yx}^2\mathbb{P}_{n+2}^t\\\partial_y^2\mathbb{P}_{n+2}^t\end{matrix}\right)=\left(\begin{matrix}
				1\,0\,0\,0\\
				0\,1\,1\,0\\
				0\,0\,0\,1
			\end{matrix}\right)\nabla^{(2)}\mathbb{P}_{n+2}^t
			\]
			and the matrix is not invertible, the families $\{\nabla^{(2)}\mathbb{P}_{n+2}^t\}_n$ and $\{\nabla^{\{2\}}\mathbb{P}_{n+2}^t\}_n$ are not orthogonal with respect to the same weight. Therefore for $m=2$  Theorem\,\ref{th1}(b) and  \cite[Theorem 4.7]{AMFPP} with $h=2$ study different families of orthogonal polynomials.
		\end{enumerate}
	\end{remark}
	{
		\begin{remark}
			In \cite{Area2018}, the author investigates families of orthogonal polynomials $\{p_n\}_n$ in one or several variables that are governed by hypergeometric structure. That is, for any nonnegative integer $n$, $p_n$ and its derivatives, differences, $q$-differences or divided-differences  satisfy equation of the same type. This leads to second-order differential, difference and q-difference equation for orthogonal polynomials in continuous variables, discrete variables and $q$-discrete variables respectively. A difference arises in the case of nonuniform lattices where bivariate Racah and bivariate $q$-Racah satisfy a fourth-order divided-difference equation of hypergeometric type. \\In our approach, for bivariate continuous orthogonal polynomials, we replace derivatives by gradients ($\nabla^{(n)}B=\nabla(\nabla^{(n-1)}B),\,n=1,2,...$) and established equivalence between  matrix version of the equation \cite[p.175]{Area2018} and four other properties (see Theorem \ref{th1}). This extends \cite[§5]{Al-Salam 1990} to orthogonal polynomials in two continuous variables.\\ It would be very interesting to formulate and prove an analogue of Theorem \ref{th1} for orthogonal polynomials in two discrete variables, two $q$-discrete variables and two nonuniform lattices. In each of these cases, an analogue of the Pearson equation (\ref{te2a}) should be found.\\
			For bivariate  discrete variables, approximating $\partial_x$ and  $\partial_y$ by the forward operators \[\Delta_1 f(x,y)=f(x+1,y)-f(x,y)\, {\rm and}\, \Delta_2 f(x,y)=f(x,y+1)-f(x,y)\] into the divergence formula $div\left(\begin{matrix}
				A\\
				B
			\end{matrix}\right)=\partial_xA+\partial_yB$, the Pearson equation (\ref{te2a}) becomes 
			\begin{equation}\label{discretepear}
				\begin{cases}
					\Delta_1(\phi_{1,1}\rho)+\Delta_2(\phi_{2,1}\rho)&=\psi_1\rho \\
					\Delta_1(\phi_{1,2}\rho)+\Delta_2(\phi_{2,2}\rho)&=\psi_2\rho.
				\end{cases}
			\end{equation}
		\end{remark}
		One can easily check that the weight function $\rho^{N,p_1,p_2}(x,y)= \frac {N!\,{p_{{1}}}^{x}{p_{{2}}}^{y} \left( 1-p_{{1}}-p_{{2}}
			\right) ^{N-x-y}}{x!\,y!\, \left( N-x-y \right) !}$ ,$x\geq 0$, $y\geq 0,\,$ $0\leq x+y\leq {N},$  $p_1>0,$ $p_2>0$ and $0<p_1+p_2<1$, (see \cite[Eq. (2.25)]{Area2018}) for bivariate Kravchuk polynomials satisfies the discrete Pearson equation (\ref{discretepear}) with $\phi_{1,1}(x,y)=\left((p_1-1)(p_2-1)-p_1p_2\right)x$, $\phi_{1,2}(x,y)={\phi_{2,1}(x,y)}=0$, $\phi_{2,2}(x,y)=\left((p_1-1)(p_2-1)-p_1p_2\right)y$, $\psi_1(x,y)=(p_2-1)x-p_1y+Np_1$ and $\psi_2(x,y)=(p_1-1)x-p_2y+Np_2$.\\
		A more general extension of our definition to nonuniform lattices (including bivariate $q$-Racah, Racah, and Askey–Wilson polynomials) will be the subject of future work. The works of Foupouagnigni, K., Nangho, and Mboutngam \cite{FKM}, together with that of K. Nangho and K. Jordaan \cite{NJ}, are expected to play an important role in this study.
	}
	
	\section{Connections and Examples}
	In this section, we build a bridge between the study of bivariate orthogonal polynomials with respect to weight functions and moment functionals, and investigate connections between the Rodrigues formula (\ref{e4b}) given by \'Alvarez de Morales et al. (cf. \cite{Miguel2009}) and (\ref{Rodrigues}), the Rodrigues formula developed in this work. We also give example of bivariate weight function that are classical in the sense of the Definition \ref{cweight}.
	
	\subsection{Connections}
	Our definition of classical orthogonal polynomials is connected to that of \cite[Section 3]{Fern} in the following way:
	\begin{proposition}\label{breach} Let $\left\{\mathbb{P}_n\right\}_n$ be a family of polynomials, orthogonal with respect to a classical weight function $\rho$ on a simply connected open subset $\Omega$ of $\mathbb{R}^2$. Let $v$ be the moment functional defined by
		\begin{equation}\label{momentf}
			\langle v,p\rangle=\int_{\Omega}p\rho dxdy,\; p\in \mathcal{P}.
		\end{equation}
		Then $\left\{\mathbb{P}_n\right\}_n$ is orthogonal with respect to 
		$v$, $v$ satisfies the Pearson-type equation $div(\Phi v)=(\psi_1,{\psi_2})v$ and det$\langle v,\,\Phi\rangle\neq 0$.
	\end{proposition}
	\begin{proof}
		\begin{equation*}
			\langle v,\,\mathbb{P}_m\mathbb{P}_n^t\rangle= \int_{\Omega}\mathbb{P}_m\mathbb{P}_n^t\rho dxdy=\delta_{m,n}H_n,
		\end{equation*}
		where $H_n$ is an invertible matrix, for $\left\{\mathbb{P}_n\right\}_n$ is orthogonal with respect to $\rho$. Therefore $\{\mathbb{P}_n\}_n$ is orthogonal with respect to $v$ (cf.  \cite[Definition 3.1]{Fern}).
		Let us prove that $div(\Phi v)=(\psi_1,{\psi_2})v$.  $div(\Phi v)=\left(\partial_x(\phi_{1,1}v)+\partial_y(\phi_{2,1}v)),\,\partial_x(\phi_{1,2}v)+\partial_y(\phi_{2,2}v))\right)$.  $\partial_x$ and $\partial_y$ here are in a weak sense, that is for functional $w$ on $\mathcal{P}$, $\langle \partial_x w,\,p\rangle=-\langle  w,\,\partial_xp\rangle$ and $\langle \partial_y w,\,p\rangle=-\langle  w,\,\partial_y p\rangle$.. Therefore, for $p\in\mathcal{P}$,
		$\langle\partial_x(\phi_{1,1}v)+\partial_y(\phi_{2,1}v),\,p\rangle=-\langle\phi_{1,1}v,\, \partial_xp\rangle-\langle\phi_{2,1}v,\partial_yp\rangle$.
		Mindful of the left multiplication of a functional by a polynomial, we have 
		$$\langle\partial_x(\phi_{1,1}v)+\partial_y(\phi_{2,1}v),\,p\rangle=-\langle v,\, \phi_{1,1}\partial_xp\rangle+\langle v,\,\phi_{2,1}\partial_yp\rangle.$$ Therefore
		$$\langle\partial_x(\phi_{1,1}v)+\partial_y(\phi_{2,1}v),\,p\rangle=-\int_{\Omega}\left(\phi_{1,1}\partial_xp+\phi_{2,1}\partial_yp\right)\rho dxdy.$$ Integration by parts of the right hand side leads to:
		
		\begin{equation}\langle\partial_x(\phi_{1,1}v)+\partial_y(\phi_{2,1}v),\,p\rangle=-\int_{\Omega}\partial_x(\rho\phi_{1,1}p)+\partial_y(\phi_{2,1}p) dxdy+\int_{\Omega}\left(\partial_x(\rho\phi_{1,1})+\partial_y(\rho\phi_{2,1})\right)p dxdy.\label{eq1}\end{equation}
		In a similar way we obtain
		\begin{equation}\langle\partial_x(\phi_{1,2}v)+\partial_y(\phi_{2,2}v),\,p\rangle=-\int_{\Omega}\partial_x(\rho\phi_{1,2}p)+\partial_y(\phi_{2,2}p) dxdy+\int_{\Omega}\left(\partial_x(\rho\phi_{1,2})+\partial_y(\rho\phi_{2,2})\right)p dxdy.\label{eq2}\end{equation} Combining \eqref{eq1} and \eqref{eq2} yields
		$$\langle div(\Phi v),p\rangle=-\int_{\Omega}div(\rho\Phi p)dxdy+\int_{\Omega}div(\rho\Phi)pdxdy.$$
		Since $\rho$ is classical, $div(\rho\Phi)=(\psi_1,\,\psi_2)\rho$,  where  $\psi_{1}$ and $\psi_{2}$ are two polynomials of degree 1. Moreover, from (\ref{bc}), \begin{equation*}
			\lim\limits_{j\rightarrow\infty}1_{\partial\Omega_j}\left(\rho \Phi\,\nabla\,u\right)\cdot\overrightarrow{n_j}=0,\; \Omega_j=\Omega\cap B(O,j),\; {\rm for~all}\; u\in\mathcal{P}.
		\end{equation*}
		Since this boundary condition is satisfied for all polynomials $u$, it will still hold when replacing $\nabla u$ by $p\nabla u$. So, \begin{equation*}
			\lim\limits_{j\rightarrow\infty}1_{\partial\Omega_j}\left(\rho \Phi\,p\nabla\,u\right)\cdot\overrightarrow{n_j}=0,\; \Omega_j=\Omega\cap B(O,j),\; {\rm for~all}\; u\in\mathcal{M}_{1,n}\left(\mathcal{P}\right)
		\end{equation*}
		which is (\ref{Neumannc}) with $A=p\Phi$. Hence, taking $A=p\Phi$, $m=1$, $M=1$, and $N=(x,y)$ in Proposition \ref{p2a}, we obtain
		$\int_{\Omega}div(\rho\Phi p)=0$. Therefore 
		$\langle div(\Phi v),p\rangle=\int_{\Omega}\rho(\psi_1,\,\psi_2)pdxdy=\langle (\psi_1,\,\psi_2)v,p\rangle$, that is $div(\Phi v)=(\psi_1,\,\psi_2)v$. As for the condition det$\langle v,\Phi\rangle\neq 0$, we take $m=0$ and $n=1$ in the structure relation (\ref{SR}) and take into account the fact that $\mathbb{P}_n$ is monic to obtain 
		\begin{equation*}
			\Phi=\left(\begin{matrix}
				\mathbb{P}_2^t&0\\
				0&\mathbb{P}_2^t
			\end{matrix}\right)A_2^{1,0}+\left(\begin{matrix}
				\mathbb{P}_1^t&0\\
				0&\mathbb{P}_1^t
			\end{matrix}\right)A_1^{1,0}+A_0^{1,0},
		\end{equation*}
		where $A_0^{1,0}$ is an invertible matrix. Multiplying both sides with $\rho$ and integrating on $\Omega$, we obtain $\int_{\Omega}\Phi\rho dxdy=A_0^{1,0}\int_{\Omega}\rho dxdy$. Since $A_0^{1,0}$ is invertible and $\int_{\Omega}\rho dxdy=\int_{\Omega}\rho \mathbb{P}_0\mathbb{P}_0^tdxdy\neq 0$, det$\int_{\Omega}\Phi\rho dxdy\neq 0$, that is $det\langle v,\Phi\rangle \neq 0$.
	\end{proof}
	\begin{corollary} If a family of orthogonal polynomials is classical  in the sense of the Definition \ref{cop}, then it is classical in the sense of \cite[Section 3]{Fern}.
	\end{corollary}
	\begin{proposition}\label{Rconnect}
		Let $\Omega$ be a simply connected subset of $\mathbb{R}^2$ and $\rho$ be a classical weight function on $\Omega$, in the sense of the Definition \ref{cweight}. Then for all $n\geq 1$, there exists an orthogonal matrix $O_n$ such that
		\[\mathbb{Q}_n^t=\tfrac{(-1)^n}{\rho}\left[div^{(n)}\left(\rho\Phi^{\otimes n}\right)\nabla^{(n)}X_n^t\prod_{j=0}^{n-1}\Lambda_{n,j}^{-1}\right]O_n,\]
		where $\mathbb{Q}_n$ is the polynomial given by the formula \cite[(35)]{Miguel2009}. 
	\end{proposition}
	\begin{proof}
		From Proposition \ref{breach}, the moment functional (\ref{momentf})  associated with $\rho$ satisfies $div(\Phi v)=(\psi_{1},\,\psi_{2})v$ and det$\langle v,\Phi\rangle\neq 0$. So it is a classical moment functional. Since  $div(\rho\Phi)=\rho(\psi_{1},\,\psi_{2})$, $\rho$ is a symmetric factor for the partial differential associated with $v$ (cf. \cite[Proposition 5.1]{Miguel2009}). Moreover, since $\Phi$ satisfies the system (\ref{te1a})-(\ref{te1b}), by means of product rule, we obtain 
		\begin{empheq}[left=\empheqlbrace]{align}
			\partial_x(\phi_{1,1}\varPhi)+\partial_y(\phi_{2,1}\partial_y\varPhi)=\varPhi P_0,\nonumber\\
			\partial_x(\phi_{1,2}\partial_x\varPhi)+\partial_y(\phi_{2,2}\varPhi)=\varPhi P_1\nonumber,
		\end{empheq}
		where  $P_0=\nabla(\phi_{1,1},\,\phi_{2,1})+I_2\partial_x\phi_{1,1}+I_2\partial_x\phi_{2,1}$ and $P_1=\nabla(\phi_{1,2},\,\phi_{2,2})+I_2\partial_x\phi_{1,2}+I_2\partial_x\phi_{2,2}$ are $2$-matrix polynomials of total degree at most 1.Therefore (cf. \cite[Theorem 6.5]{Marcellan}) the family of polynomials $\mathbb{Q}_n^t=\frac{1}{\rho}div^{\{n\}}\left(\rho\Phi^{\{n\}}\right)$ is orthogonal with respect to $v$. $div^{\{n\}}$ is given by (\ref{mdiv}) and $\Phi^{\{n\}}$ is the $n$-th second Kronecker product of $\Phi$. Let $\left\{\mathbb{P}_n\right\}_n$ be the family of monic polynomials, orthogonal with respect to $\rho$. From Theorem \ref{th1}, 
		\[\mathbb{P}_n^t=\tfrac{(-1)^n}{\rho}div^{(n)}\left(\rho\Phi^{\otimes n}\right)\nabla^{(n)}\mathbb{P}_n^t\prod_{j=0}^{n-1}\Lambda_{n,j}^{-1}.\] Since $\mathbb{P}_n$ is monic $\mathbb{P}_n^t=X^t_n+X_{n-1}^tG_{n,n-1}\dots$. Therefore $\nabla^{(n)}\mathbb{P}_n^t=\nabla^{(n)}X_n^t$ and 
		\[\mathbb{P}_n^t=\tfrac{(-1)^n}{\rho}div^{(n)}\left(\rho\Phi^{\otimes n}\right)\nabla^{(n)}X_n^t\prod_{j=0}^{n-1}\Lambda_{n,j}^{-1}.\] Since $\left\{\mathbb{P}_n\right\}_n$ is orthogonal with respect to the functional $v$, there exist an orthogonal matrix $O_n$ such that (cf.\cite[Theorem 3.2.14]{Yuan}) $\mathbb{Q}^t_n=\mathbb{P}^t_nO_n$, $n\geq 1$.
		
	\end{proof}
	
	{\begin{remark}\label{pearrem}
			The Rodrigues formula (\ref{Rodrigues}) is obtained by combining (\ref{te3}) and (\ref{te2b}) (see Step 2 of the proof of the Theorem \ref{th1}). Observing that (\ref{te2b}) itself is established by induction using the Pearson equation (\ref{peaeq}) and the differential system (\ref{te1a})-(\ref{te1b}), and (\ref{te2b}) can be obtained by applying the successive gradient $\nabla^{(m)}$ on the equation
			\begin{equation}\label{deq0}
				\left(\left(\varPhi\nabla\right)\cdot\nabla\right)\mathbb{P}_{n+m}^t+\psi_1\partial_x\mathbb{P}_{n+m}^t+{\psi_2}\partial_y\mathbb{P}_{n+m}^t+\mathbb{P}_{n+m}^t\Lambda_n=0,
			\end{equation}
			we deduce the following: If a family of orthogonal polynomials in two variables $\{\mathbb{P}_n\}_n$ is a solution of (\ref{deq0}), where $\Phi\in\mathcal{M}_2\left(\mathcal{P}\right)$ is a symmetric matrix such that there exist $\rho$, a solution of the Pearson equation $div(\rho\Phi)=\rho(\psi_1,\,\psi_2)$, then $\left\{\mathbb{P}_n\right\}_n$ satisfy the Rodrigues formula (\ref{Rodrigues}).
	\end{remark}}
	
	\subsection{Examples}
	\subsubsection{ Orthogonal polynomials on a triangle}
	We prove that a family of orthogonal polynomials  with respect to the weight function \cite[Eq. (2.4.1)]{Miguel2009} $\rho(x,y)=x^{\alpha}y^{\beta}(1-x-y)^{\gamma}$,$ \, \alpha,\,\beta,\,\gamma>-1$ defined on the triangle
	$$\Omega=\left\{(x,y)\in\mathbb{R}^2;x,y\geq 0\,\text{and}\,x+y\leq 1\right\}$$ is classical in our sense and give, for this family, the matrix $\Lambda_{n,j}$
	involved in the Rodrigues formula developed in this work.\\
	After straightforward computations
	\begin{equation*}
		div(\varPhi\,\rho)=\varPsi^t\,\rho,\text{with}\,\varPhi=\left(\begin{matrix}
			x(1-x) & -xy\\
			-xy & y(1-y)
		\end{matrix}\right), \varPsi=\left(\begin{matrix}
			-\left( \alpha+\beta+\gamma+3\right) x+\alpha+1\\
			-\left( \alpha+\beta+\gamma+3\right) y+\beta+1
		\end{matrix}\right)
	\end{equation*} 
	and the matrix $\Phi$ satisfies the system (\ref{te1a})-(\ref{te1b}).
	Therefore, for the weight function $\rho(x,y)=x^{\alpha}y^{\beta}(1-x-y)^{\gamma}$,$ \, \alpha,\,\beta,\,\gamma>-1$ to be classical in the sense of  the Definition\ref{cweight}  it remains for us to prove that the boundary condition (\ref{bc}) is fulfilled.\\ Since $\Omega$ here is a triangle, there is $j_0\geq 0$ such that for $j\geq j_0$, $\Omega\subset B(O,j)$. Therefore for $j\geq j_0$, $\Omega_j=\Omega\cap B(O,j)=\Omega$ and $$\partial\Omega_j=\partial\Omega=\left\{(x,0);0\leq x\leq 1\right\}\cup\left\{(0,y);0\leq y\leq 1\right\}\cup\left\{(x,y);0\leq x,y,\,\text{and}\,x+y= 1\right\}.$$
	So, let $(x,y)\in \partial\Omega_j$, $j\geq j_0$.
	\begin{itemize}
		\item If $(x,y)\in \left\{(x,0);\,0\leq x\leq 1\right\}$ then, the outward normal at $(x,y)$ is $\overrightarrow{n}_j=(0,1)^t$. Observing that
		$\rho\Phi\nabla u=\rho(x,y)(x(1-x)\partial_xu-xy\partial_yu,-xy\partial_xu+y(1-y)\partial_yu)^t$, we have
		$(\rho\Phi\nabla u)\cdot\overrightarrow{n}_j=(1-x-y)^{\gamma}(-x^{\alpha+1}y^{\beta+1} +x^{\alpha}y^{\beta+1}(1-y)\partial_yu)$. Letting $y$ go to $0$ and taking into account the fact that $\beta>-1$, we obtain $(\rho\Phi\nabla u)\cdot\overrightarrow{n}_j=0$. Thus $\lim_{j\rightarrow\infty}(\rho\Phi\nabla u)\cdot\overrightarrow{n}_j=0$. In a similar way, we prove that on  $\left\{(0,y);\,0\leq y\leq 1\right\}$,  $\lim_{j\rightarrow\infty}(\rho\Phi\nabla u)\cdot\overrightarrow{n}_j=0$.
		\item If $(x,y)\in \left\{(x,y);\,0\leq x,\, y,\:x+y=1\right\}$, then $\overrightarrow{n}_j=\tfrac{1}{\sqrt{2}}(1,1)$ and\\
		$\left(\rho(x,y)\Phi\nabla u\right)\cdot \overrightarrow{n}_j=\tfrac{1}{\sqrt{2}}x^{\alpha}y^{\beta}(1-x-y)^{\gamma+1}(x\partial_xu+y\partial_yu)$ which is equal to $0$, for $x+y-1=0$ and $\gamma>-1$. Hence $\lim_{j\rightarrow\infty}(\rho\Phi\nabla u)\cdot\overrightarrow{n}_j=0$.
	\end{itemize}

	For this family of polynomials, we obtained after computation the matrix  $\Lambda_{n,m}$ involved in the Rodrigues Formula (Theorem \ref{th1} (d)) as 
	$\Lambda_{n,j}=(n-j)(n+j+2+\alpha+\beta+\gamma)I_{n+1}$. Assuming the polynomial $\mathbb{P}_n$ to be monic, $\mathbb{P}^t_n=X_n^t+X_{n-1}^tG_{n,n-1}+\dots$, $\nabla^{(n)}\mathbb{P}_n^t=\nabla^{(n)}X_n.$ Therefore the Rodrigues formula Theorem \ref{th1}(d) reads
	\[\mathbb{P}_n^t(x,y)=\frac{(-1)^n}{n!(\alpha+\beta+\gamma+n+2)_nx^{\alpha}y^{\beta}\left(1-x-y\right)^{\gamma}}div^{(n)}\left(x^{\alpha}y^{\beta}\left(1-x-y\right)\Phi^{\otimes n}\right)\nabla^{(n)}X_n^t.\]
	
	From this formula, we can obtain explicit expansion of $\mathbb{P}_0^t$, $\mathbb{P}_1^t$, $\mathbb{P}_2^t$ and $\mathbb{P}_3^t$ as follows
	\begin{eqnarray*}
		\mathbb{P}_0^t(x,y)&=&1, \\
		\mathbb{P}_1^t(x,y)&=&(x-{\tfrac {\alpha+1}{3+\alpha+\beta+\gamma}},y-{\tfrac {\beta+1}{3+\alpha+\beta+\gamma}}),\\
		\mathbb{P}_2^t(x,y)&=&(x^2-{\tfrac { 2\left( \alpha+2 \right) }{5+\alpha+\beta+\gamma}}x+{
			\tfrac {\left(\alpha+1\right)_2}{ \left( 4+\alpha+\beta+\gamma \right)_2 }},xy-{\tfrac { \left( \beta+1 \right) x}{5+\alpha+\beta+\gamma}}-{\tfrac {
				\left( \alpha+1 \right) y}{5+\alpha+\beta+\gamma}}+{\tfrac {\left(\alpha+1\right)\left(\beta+1\right)}{ \left( 4+
				\alpha+\beta+\gamma \right)_2 }},\\
		&&y^2-{\tfrac { 2\left( \beta+2 \right)}{5+\alpha+\beta+\gamma}}y+{\tfrac 
			{\left(\beta+1\right)_2}{ 
				\left( 4+\alpha+\beta+\gamma \right)_2 }}),\\
		\mathbb{P}_3^t(x,y)&=&(p_{03},p_{12},p_{21},p_{30}),
	\end{eqnarray*}
	where
	\begin{eqnarray*}
		p_{30}(x,y)&=&x^3-{\tfrac{ 3\left( \alpha+3 \right)}{7+\alpha+\beta+\gamma}}{x}^{2}+
		{\tfrac { 3\left(\alpha+2\right)_2}{ \left( 6+
				\alpha+\beta+\gamma \right)_2  }}x-
		{\tfrac{\left(\alpha+1\right)_3}{\left(5+\alpha+\beta+\gamma \right)_3 }}\\
		p_{21}(x,y)&=&x^2y-{\tfrac { \left( \beta+1 \right) }{7+\alpha+\beta+\gamma}}{x}^{2} -{\tfrac { 2\left( \alpha+2 \right)}{7+\alpha+\beta+\gamma}xy
		}+{\tfrac {2\left(\alpha+2\right)\left(\beta+1\right)}{ \left( 6+\alpha+\beta+
				\gamma \right)_2}}x+{
			\tfrac { \left( {\alpha}+1 \right)_2}{ \left( 6+\alpha+
				\beta+\gamma \right)_2}}y-{\tfrac {
				\left(\alpha+1\right)_2\left(\beta+1\right)}{
				\left( 5+\alpha+\beta+\gamma \right)_3 }}\\
		p_{12}(x,y)&=&xy^2-{\tfrac {2\left( \beta+2 \right)}{7+\alpha+\beta+\gamma}}xy-{\tfrac {
				\left( \alpha+1 \right)}{7+\alpha+\beta+\gamma}}{y}^{2}
		+{\tfrac {\left(\beta+1\right)_2}{ \left( 6+\alpha+\beta+\gamma
				\right)_2 }}x+{\tfrac {2
				\left( \alpha+1 \right)\left( \beta+2\right) }{ \left( 6+\alpha+
				\beta+\gamma \right)_2 }}y-{\tfrac {\left(\alpha+1\right)\left(\beta+1\right)_2}
			{\left( 5+\alpha+\beta+\gamma \right)_3 }}\\
		p_{03}(x,y)&=&y^3-{\tfrac { 3\left( \beta+3 \right) }{7+\alpha+\beta+\gamma}}{y}^{2}+{\tfrac { 3\left(\beta+2 \right)_2 y}{ \left( 6+\alpha+
				\beta+\gamma \right)_2 }}-{\tfrac {\left(\beta+1\right)_3}{\left( 5+\alpha+\beta+
				\gamma \right)_3 }},
	\end{eqnarray*}
	where $(a)_n$ is the Pochhammer symbol defined by $(a)_n=a(a+1)\dots(a+n-1)$, $n>1$, $(a)_0=1$ with $a\in\mathbb{C}$ and $n=1,2,\dots\, .$
	
	It can be verified by direct computation that polynomials  $\mathbb{P}_0^t$, $\mathbb{P}_1^t$, $\mathbb{P}_2^t$ and $\mathbb{P}_3^t$ generated above satisfy the partial differential equation \cite[p.38]{Yuan}, for orthogonal polynomials on triangle,  with parameters $(\alpha, \beta,\,\gamma)$ replaced by 
	$(\alpha+\tfrac{1}{2}, \beta+\tfrac{1}{2},\,\gamma+\tfrac{1}{2})$.
	\begin{remark} The product of Hermite polynomials, the product of Laguerre polynomials, the product of Hermite and Laguerre polynomials,  and the product of Jacobi polynomials are also examples of classical orthogonal polynomials in our sense. 
	\end{remark}
	
	\subsubsection{Example 2:  Intriguing case (cf. Krall and Sheffer \cite[p. 362]{Krall}  and Littlejohn \cite{Littlejohn})} 
	{Consider the equation}{ 
		\begin{equation}\label{ks}
			3y\partial_x^2Y+2\partial_{xy}^2Y-x\partial_xY-y\partial_yY+nY=0.
		\end{equation}
		The explicit expression of $\Phi$ and $\Psi$ are 
		\begin{equation*}
			\varPhi=\left(\begin{matrix}
				3y & 1\\
				1 & 0
			\end{matrix}\right), \varPsi=\left(\begin{matrix}
				-x\\
				-y
			\end{matrix}\right)
		\end{equation*} 
		Krall and Sheffer  \cite[p.362]{Krall} proved that (\ref{ks}) has orthogonal polynomial sequences as solutions. The domain of orthogonality of this family of polynomial is not known. Moreover, the family  cannot be a positive definite orthogonal polynomial system (see \cite[Example 4.2]{Kim1997}) on a domain. In addition, det~$\Phi=1<0$, that is the matrix $\Phi$ is not positive definite. So the family is not classical in our sense. However observing that the matrix $\Phi$ satisfies the differential system  (\ref{te1a})-(\ref{te1b}) and the function $\rho(x,y)=\exp(y^3-xy)$ satisfies the Pearson equation (\ref{peaeq}), we derive from Remark \ref{pearrem} that the system of orthogonal vector polynomials $\left\{\mathbb{P}_n\right\}_n$, solution to (\ref{ks}), satisfies the Rodrigues formula (\ref{Rodrigues}) with ${\Lambda_{n,j}=(n-j)I_{n+1},j=0,1,...\,}$.
		Therefore, 
		\begin{equation*}
			\mathbb{P}_n^t(x,y)=\frac{(-1)^n}{{n!}\exp{(y^3-xy)}}div^{(n)}\left(\exp{(y^3-xy)}\Phi^{\otimes n}\right)\nabla^{(n)}X_n^t.
		\end{equation*}
		So, explicit expressions of  $\mathbb{P}_0,\,\mathbb{P}_1,\,\mathbb{P}_2,$ and $\mathbb{P}_3$.
		\begin{eqnarray*}
			\mathbb{P}_0(x,\,y)&=&1\\
			\mathbb{P}_1(x,\,y)&=&[  x,\,y]\\
			\mathbb{P}_2(x,\,y)&=& [-6\,y+{x}^{2},\,-1+xy,\,{y}^{2}]\\
			\mathbb{P}_3(x,\,y)&=&[-18\,xy+12+{x}^{3},\,-6\,{y}^{2}-2\,x+{x}^{
				2}y,\,y \left( -2+xy \right),\,{y}^{3}]\\
			\mathbb{P}_4(x,\,y)&=&[108\,{y}^{2}-36\,{x}^{2}y+48\,x+{x}^{4},\,
			30\,y-18\,x{y}^{2}-3\,{x}^{2}+{x}^{3}y,\,-6\,{y}^{3}+2-4\,xy+{x}^{2}{y}^
			{2},\\
			&&\,{y}^{2} \left( -3+xy \right),\,{y}^{4}].
		\end{eqnarray*}
		It can be verified by direct computation that polynomials  $\mathbb{P}_0^t$, $\mathbb{P}_1^t$, $\mathbb{P}_2^t$, $\mathbb{P}_3^t$ and $\mathbb{P}_4^t$ generated above satisfy the partial differential equation (\ref{ks}).
		From this example we have the following remark.}
	
	{
		\begin{remark}
			It is not sufficient for a family of orthogonal polynomials in two variables to satisfy the Rodrigues formula (\ref{Rodrigues}) stated in Theorem \ref{th1} in order to be classical in our sense. 
	\end{remark}}

	\section{Concluding remarks}
	In this paper we characterize polynomials in two {continuous} variables, orthogonal with respect to weight function satisfying (\ref{e5a}), the boundary condition (\ref{bc}) and the differential system (\ref{te1a})-(\ref{te1b}). We establish five equivalent properties and derive a definition of classical orthogonal polynomials (COP) in two variables as well as a bridge between the theory of COP with respect to weight functions and that based on moment functionals (see Proposition \ref{breach}).   The fact that the successive tensor product  $\Phi^{\otimes m},m\geq 1$ involved in the characterization theorem is symmetric and positive definite as $\Phi$ ensures the orthogonality of successive gradient of transpose of COP with respect to a weighted matrix. Moreover, the compact form of (\ref{te2b}) highlights the difference from (\ref{inieq}). The  Rodrigues formula (\ref{Rodrigues}) developed in this work is obtained by iterating the closed form of (\ref{te2b}) and is up to a multiplicative matrix factor equal to the one in \cite{Miguel2009}. The structure relation for successive gradients involve three terms which makes it different from the relation (\ref{FSR1}) developed in \cite{AMFPP}.  The properties appearing in the main theorem of this work are an analogue of those in {\cite[§5]{Al-Salam 1990}} for COP in one variable.  A natural question that arises is whether a classification of all the known families of orthogonal polynomials in two variables, that satisfy our definition, is possible. Work on an analysis along these lines is in progress. 
	
	\section{Acknowledgment} The first author would like to express his appreciation to the organizers of OPSFA 2019 during which the idea of this paper was conceived. The authors are grateful to the University of South Africa for an International General Research grant which supported a research visit by the first author to the University of South Africa for the period 1 Oct 2021--11 Dec 2021. Authors thank referees for giving many valuable suggestions.

	
	
	

\end{document}